\documentclass[review]{siamart1116}

\usepackage{latexsym}
\usepackage{amsmath}
\usepackage{amssymb}
\usepackage{amsfonts}
\usepackage[latin1]{inputenc}
\usepackage{mathrsfs}
\usepackage{amsfonts}
\usepackage{graphicx}
\usepackage{colortbl,dcolumn}
\usepackage{amsmath}
\usepackage{amsfonts}
\usepackage{graphicx}
\usepackage{algorithmic}
\ifpdf
  \DeclareGraphicsExtensions{.eps,.pdf,.png,.jpg}
\else
  \DeclareGraphicsExtensions{.eps}
\fi

\newtheorem{tm}{Theorem}[section]
\newtheorem{rk}{Remark}[section]

\newtheorem{prop}{Proposition}[section]
\newtheorem{lm}{Lemma}[section]
\newtheorem{cor}{Corollary}[section]

\newcommand{\E}{\mathbb E}
\newcommand{\PP}{\mathbb P}
\newcommand{\N}{\mathbb N}
\newcommand{\R}{\mathbb R}

\newcommand{\Z}{\mathbb Z}
\newcommand{\bi}{\mathbf i}
\newcommand{\bs}{\mathbf s}

\newcommand{\LL}{\mathcal L}

\newcommand{\HH}{\mathbb H}

\newcommand{\FFF}{\mathscr F}
\newcommand{\<}{\langle}
\renewcommand{\>}{\rangle}

\ifpdf
  \DeclareGraphicsExtensions{.eps,.pdf,.png,.jpg}
\else
  \DeclareGraphicsExtensions{.eps}
\fi

\newcommand{\TheTitle}{Analysis of A Splitting Scheme for  Damped Stochastic Nonlinear Schr\"odinger Equation with Multiplicative Noise} 
\newcommand{\TheAuthors}{Jianbo Cui and  Jialin Hong}

\headers{}{\TheAuthors}

\title{{\TheTitle}\thanks{Submitted to the editors in DATE.
\funding{This work was supported by National Natural Science Foundation of China (No. 91630312, No. 91530118 and No. 11290142).}}}

\author{
Jianbo Cui 
\thanks{1. LSEC, ICMSEC, 
Academy of Mathematics and Systems Science, Chinese Academy of Sciences, Beijing,  100190, China\qquad
2. School of Mathematical Science, University of Chinese Academy of Sciences, Beijing, 100049, China 
(\email{jianbocui@lsec.cc.ac.cn}(corresponding author), \email{hjl@lsec.cc.ac.cn})}
\and 
Jialin Hong \footnotemark[2]
}

\usepackage{amsopn}

\ifpdf
\hypersetup{
  pdftitle={\TheTitle},
  pdfauthor={\TheAuthors}
}
\fi

\usepackage{ulem}
\usepackage{subfigure}
\usepackage{graphicx}
\usepackage{enumerate}
\usepackage{amsmath,bm}
\usepackage{latexsym}
\usepackage{float}

\begin{document}

\maketitle

\begin{abstract}
In this paper, we investigate  the damped stochastic nonlinear Schr\"odinger(NLS) equation with multiplicative noise and its splitting-based approximation.
When the damped effect is large enough,  we prove that the solutions of both the damped stochastic NLS equation and 
the splitting scheme
are exponentially stable and possess some exponential integrability.  These properties show that the strong order of the scheme is $\frac 12$ and independent of time. Additionally, we analyze the regularity of the Kolmogorov equation with respect to the stochastic nonlinear Schr\"odinger  equation. As a consequence,  the weak order of the scheme is shown to be 1 and independent of time. 
\end{abstract}


\begin{AMS}
{60H35}, {35Q55, 60H15, 65M12.}
\end{AMS}

\section{Introduction}
In many fields of economics and the natural sciences,  stochastic partial differential equations (SPDEs) play important roles. 
Since many SPDEs  can only be solved numerically,  it is a crucial research problem to construct and study discrete numerical approximation schemes which converge with strong and weak convergence rates to the solutions of  such SPDEs.
For SPDEs with monotone coefficients, there exist fruitful results on strong error analysis of temporal and spatial numerical approximations (see, e.g., \cite{ ACLW16,  BAK16, CHL16a, BD06, Gyon99, GK03}). 
However, there exists only a few results in the scientific literature which establish strong and weak convergence rates for a time discrete approximation scheme in the case of an SPDE with a nonglobally monotone nonlinearity 
(see, e.g., \cite{ CHL16b, CHLZ17, Dor12, HJ14, HJW13, JP16}). This motives us to construct strong and weak approximations for this kind of SPDE.

The stochastic nonlinear Schr\"odinger (NLS) equation, as a representative SPDE,  models the propagation of nonlinear dispersive waves in inhomogeneous or random media (see, e.g., \cite{BCI95}). 
In \cite{BD99} and \cite{BRZ16, BD03}  it was proved that the stochastic NLS equation admits a unique solution in $\HH$ and $\HH^1$, respectively.
Recently, \cite{CHP16, CHL16b} gave the global well-posedness  of the one-dimensional stochastic NLS equation in $\HH^2$.
In this paper, we focus on strong and weak approximations of the following one-dimensional  damped stochastic nonlinear equation with multiplicative noise:
\begin{align}\label{dnls}
du&=(\bi \Delta u +\bi \lambda |u|^2u-\alpha u)dt+\bi udW(t)  \;\; \text{in} \;\;  \mathbb R \times (0,\infty); \\\nonumber
u(0)&=u_0 \;\; \text{in} \;\; \mathbb R ,
\end{align}
where  
 $\lambda=1$ or $-1$ corresponds to focusing or defocusing cases, respectively, and $\alpha(\cdot)$ is  a real-valued function. 
When studying the propagation of waves over very long distance in random media, the damping term $-\alpha u$ cannot be neglected (see e.g. \cite{DO05}).
The diffusion term represents the fluctuation effect of a physical process in random media, where $W=\{W(t):\ t\in [0,T]\}$ is an $L^2(\mathbb R; \R)$-valued $Q$-Wiener process on  a stochastic basis $(\Omega, \FFF, \FFF_t, \PP)$; i.e., there exists an orthonormal basis $\{e_k\}_{k\in \N_+} $ of $L^2(\mathbb R; \R)$ and a sequence of mutually independent, real-valued Brownian motions $\{\beta_k\}_{k\in \N_+} $ such that $W(t)=\sum\limits_{k\in \N_+} Q^{\frac12}e_k\beta_k(t)$, $t\in [0,T]$.

There have been many works concentrating on construction and analysis of numerical approximations for the stochastic NLS equation. Paper \cite{BD06} studies 
a type of Crank--Nicolson  semidiscrete schemes and  shows that for stochastic NLS equation with Lipschitz coefficients, these Crank--Nicolson  type schemes have strong order $\frac 12$ in general and order $1$ if the noise is additive and that the weak order is always 1. 
In order to inherit the symplectic structure of the stochastic NLS equation, \cite{CH16} studies symplectic Runge--Kutta methods and obtains the convergence theorem for the Lipschitz cases. 
Paper \cite{AC16} studies an explicit  exponential scheme and shows that  it preserves the trace formula for stochastic linear NLS equation with additive noise.
For a stochastic NLS equation with non-Lipschitz or 
nonmonotone coefficients, some papers have constructed   strong numerical approximations  and obtained convergence rates in a certain sense such as pathwise or in probability weaker than in strong sense (see, e.g., \cite{CHP16, BD06, Liu13a} and references therein). 
Progress has been made in \cite{CHL16b, CHLZ17}, where the authors obtained  strong convergence rates of the spatial centered difference method, the spatial Galerkin method and a temporal splitting method  for a conservative stochastic NLS equation.

In this article, we 
apply the splitting ideas  in \cite{CHLZ17, GK03} to approximating  \eqref{dnls} and aim to show  the strong and weak order of this
splitting scheme. 
The key to obtaining strong and weak convergence rates of numerical schemes for SPDEs with nonmonotone coefficients is to obtain some a priori estimates and  exponential integrability of exact and numerical solutions (see, e.g., \cite{CHJ13,  CHL16b, CHLZ17,HJ14,HJW13,JP16}). 
On the one hand, we prove some a priori estimations of the exact solution of \eqref{dnls}, as well as those of the numerical solution, to get the time-independent strong error estimation. As a consequence, the solution of  \eqref{dnls} is shown to be exponentially stable.
On the other hand, we show the exponential integrability properties of  exact and numerical solutions by an exponential integrability lemma established in \cite[Corollary 2.4]{CHJ13}; see also \cite[Lemma 3.1]{CHL16b}.
This type of exponential integrability is also useful to get the strongly continuous dependence  on initial data of both  exact and numerical solutions and to deduce Gaussian tail estimations of these solutions (see e,.g. \cite{CHJ13, CHL16b, CHLZ17}).
To obtain the weak convergence order of the proposed scheme, we study  the regularity of the transformed  Kolmogorov equation of the damped  stochastic NLS equation with nonmonotone coefficient. Based on this regularity result, we prove that  the weak order of the proposed scheme is first order and independent of time. To the best of our knowledge, this is the first  weak convergence order result of temporal approximations for the stochastic NLS equation with nonmonotone coefficients driven by multiplicative noise.


The rest of this paper is organized as follows. 
In Section 2, we prove that the damped stochastic NLS equation is exponentially stable and exponentially integrable. Section 3 is devoted to obtaining some a priori estimates of the numerical solution in Sobolev norms. Then the time-independent strong error of the solutions is given. In Section 4, we study the regularity 
of the corresponding Kolmogorov equation with respect to 
the damped stochastic NLS equation. Then we show that the weak order of the scheme is first order and time-independent.
\section{Some properties for  damped  stochastic NLS equation}

We first introduce some frequently used notation and assumptions.
The norm and inner product of $\HH:=L^2(\mathbb R; \mathbb C)$ are denoted by $\|\cdot\|$ and $\<u, v\>:=\Re\left[\int_{\mathbb R} \overline u(x) v(x) dx\right]$, respectively. $L^p:=L^p(\mathbb R; \mathbb C)$, $p\ge 1$, is the corresponding Banach space. Throughout this paper, we assume that $T$ is a fixed positive number, $u_0 \in \HH^{\bs} $ is a deterministic function and $Q^\frac12 \in \LL_{2}^\bs$ with $\bs $ being a nonnegative integer, i.e.,
\begin{align*}
\|Q^\frac12\|_{\LL_2^\bs}^2
:=\sum_{k\in \N_+}  \|Q^\frac12 e_k\|_{\HH^{\bs}}^2<\infty,
\end{align*}
where $\{e_k\}_{k\in \N_+} $ is any orthonormal basis of $L^2(\mathbb R; \R)$ and $\HH^\bs:= \HH^\bs (\mathbb R; \mathbb C)$ is the usual Sobolev space. In this paper,  $a$ and $b$ are positive numbers.
We use $C$ and $C'$ to denote  generic constants, independent of the time step size 
$\tau$, which differ from one place to another.
{In some places of this paper, the computations are formal but could be justified rigorously by  truncated techniques and  approximation arguments (see e.g. \cite{BD06}).

For damped stochastic NLS equations with additive noise, \cite{DO05} studies the long-time behavior of its solution and obtains the ergodicity of the weakly damped NLS equation. It is natural to study the long-time behaviors of the damped stochastic NLS equation with multiplicative noise, i.e., \eqref{dnls}.
In this section, we want to investigate the mutual influence among the damping effect, the cubic nonlinearity and the noise intensity and further study the long-time behaviors. This is  our other  motivation  for considering Eq. \eqref{dnls}. It should be mentioned that  when $\alpha(x)=\frac 12 F_Q(x):=\frac 12\sum_{k=1}^{\infty}(Q^{\frac 12}e_k)^2(x)$, the stochastic NLS equation \eqref{dnls} has the conserved quantity charge (see \cite{BD03}), i.e., $\|u(t)\|^2=\|u_0\|^2$, $t<T$, a.s. 
Next, we mainly focus on some a priori estimates and long-time behaviors of the exact solution for  \eqref{dnls}.

\begin{lm}\label{char}
 Let $\|\alpha\|_{L^{\infty}}<\infty$, $Q^{\frac 12}\in \LL_{2}^1$, and $u_0 \in \HH$. Then $\|u\|$ is bounded a.s. in any finite interval $[0,T]$.
Moreover, if $\sup\limits_{x\in\mathbb R}(\frac12 F_Q(x)- \alpha(x))\le 0$, then the upper bound is independent of   
T. 
\end{lm}
\begin{proof}
By the It\^o formula, we have 
\begin{align*}
\frac 12\|u(t)\|^2
&=\frac 12\|u_0\|^2
+\int_0^t \<u,\bi \Delta u+\bi\lambda |u|^2u-\alpha u\>ds\\
&\qquad+\int_0^t \<u,\bi u dW(s)\>
+\int_0^t\frac 12\sum_{k}\<\bi uQ^{\frac 12}e_k, \bi uQ^{\frac 12}e_k\>ds\\
&=\frac 12\|u_0\|^2
+\int_0^t\int_{\mathbb R}|u|^2(\frac12  F_Q-\alpha)dxdt.
\end{align*}
\end{proof}
The Sobolev embedding theorem, $Q^{\frac 12}\in \LL^{1}_2$ and $\|F_Q\|_{L^{\infty}}<\infty$, combined with  $\|\alpha\|_{L^{\infty}}<\infty$,
imply that
\begin{align*}
\frac 12\|u(t)\|^2
&\le \frac 12\|u_0\|^2
+ \int_0^t\|u\|^2(\frac12 \|F_Q\|_{L^{\infty}}+\|\alpha\|_{L^{\infty}})dt.
\end{align*}
Then Gronwall inequality yields that
\begin{align*}
\|u(t)\|^2
&\le \exp\big( T\|F_Q\|_{L^{\infty}}+2T\|\alpha\|_{L^{\infty}}\big)\|u_0\|^2.
\end{align*}
When  $\sup\limits_{x\in\mathbb R}(\frac12 F_Q(x)- \alpha(x))\le 0$, a similar argument yields that
\begin{align*}
\|u(t)\|^2
&\le\|u_0\|^2.
\end{align*}

\begin{cor}\label{cor-char}
If in addition, $\sup\limits_{x\in\mathbb R}(\frac12 F_Q(x)- \alpha(x))\le -a$,
then the charge is exponentially stable.
\end{cor}
\begin{proof}
Similar to Lemma \ref{char}, we  obtain 
\begin{align*}
\|u(t)\|^2\le \|u_0\|^2-2a\int_0^t\|u(s)\|^2ds,
\end{align*}
which yields that 
\begin{align}\label{char-exp}
\|u(t)\|^2
\le \exp(-2at)\|u_0\|^2.
\end{align}
\end{proof}

When $\alpha(x)=\frac 12F_Q(x)$,   \eqref{dnls} becomes the stochastic NLS equation with a 
conserved  quantity: charge. One cannot expect the 
following long-time behaviors of the exact solution in this conserved case.
When $\alpha(x)=a+\frac 12F_Q(x)$, \eqref{dnls} satisfies the condition of Corollary \ref{char} and thus the charge is exponentially decaying.
The above results inspire us to consider the long-time behavior of $u$, such as its corresponding invariant measure and ergodicity. Actually, direct calculation yields that the Dirac measure at $0$ is one of the invariant measures. The uniqueness  of the  invariant can be obtained as follows.

\begin{prop}
	\label{h1}
	Assume that $\alpha  \in \HH^1$,  $\sup\limits_{x\in\mathbb R}(\frac12 F_Q(x)- \alpha(x))\le -a$,  $u_0\in \HH^1$ and $Q^{\frac 12}\in \LL_2^{1 }$. For any $p\ge 2$, we have 
	\begin{align*}
	\sup_{t\in[0,\infty)}\E\Big[\big\|u(t)\big\|_{\mathbb H^1}^p\Big]
	\le C(1+\|u_0\|^p_{\mathbb H^{ 1}}+\|u_0\|^{3p}).
	\end{align*}
\end{prop}
\begin{proof}
For simplicity, we only prove the case $p=2$. 
One can apply the It\^o formula to the  appropriate power of the energy  functional $H(u):\frac 12\|\nabla u\|^2-\frac \lambda 4\|u\|_{L^4}^4$ and apply the
Burkholder--Davis--Gundy inequality to get the desired result for $p>2$.
Similar to \cite{CHL16b},
thanks to Gagliardo--Nirenberg inequality $\|u\|_{L^4}^4\le2\|\nabla u\|\|u\|^3$, we need only prove the uniform boundedness of the energy functional $H(u(t))$. The
It\^o formula yields that  
\begin{align*}
&\mathbb E [H(u(t))]-H(u_0)\\
&=\int_0^t \E \Big[\big\<\nabla u,\nabla u (\frac {F_Q}2-\alpha)\big\>\Big]ds+
\int_0^t\E \Big[\big \<\nabla u, u  (\sum_{k}Q^{\frac 12}e_k \nabla Q^{\frac 12}e_k- \nabla\alpha)\big\>\Big]ds\\
&\quad +\int_0^t \frac 12\sum_k\E \Big[\big\<u, u |\nabla Q^{\frac 12}e_k|^2 \big\>\Big] ds
+\int_0^t \lambda\E \Big[\big\<|u|^2u, u(\alpha-\frac {F_Q}2) \big\>\Big]ds.
\end{align*}
The H\"older, Gagliardo--Nirenberg and Young inequalities  and Sobolev embedding theorem imply that for  $a>\epsilon>0$,
 \begin{align*}
 \E [H(u(t))]&\le H(u_0)-(a-\epsilon)\int_0^t\E\big[\|\nabla u\|^2\big]ds
+C(\epsilon)\int_0^t\E \bigg[\|u\|^2\big(\|Q^{\frac 12}\|^4_{\LL^1_2}\\
&\quad +\|Q^{\frac 12}\|^8_{\LL^1_2}+\|\nabla \alpha\|^4\big)+\|u\|^6\|\alpha-\frac {F_Q}2\|^2_{L^{\infty}}\bigg]ds.
 \end{align*}
By the fact that 
$\frac 12\|\nabla u\|^2 -\frac 14\|u\|_{L^4}^4\le H(u) \le \frac 12\|\nabla u\|^2+\frac 14\|u\|_{L^4}^4$ and the Young inequality, we have for small $\eta>0$,
\begin{align*}
 \E [H(u(t))]
&\le  H(u_0)-\frac {2(a-\epsilon)}{1+\eta}\int_0^t\E\big[H(u)\big]ds+C(\epsilon,\eta)\int_0^t\E \bigg[\|u\|^2\big(
\|Q^{\frac 12}\|_{\LL^1_2}^4\\
&\quad+\|Q^{\frac 12}\|_{\LL^1_2}^8+\|\nabla \alpha\|^4\big)+\|u\|^6\big(1+\|\alpha-\frac {F_Q}2\|^2_{L^{\infty}}\big)\bigg]ds.
\end{align*}
The Gronwall inequality, together with the charge evolution law in Corollary \ref{cor-char}, yields that 
\begin{align}\label{exp-dec}
 &\E [H(u(t))]\le e^{-\frac {2(a-\epsilon)}{1+\eta}t}H(u_0)
+C(\epsilon,\eta)e^{-\frac {2(a-\epsilon)}{1+\eta}t}\int_0^t \bigg(e^{(\frac {2(a-\epsilon)}{1+\eta}-2a)s}\|u_0\|^2\big(\|Q^{\frac 12}\|_{\LL^1_2}^4\\ \nonumber 
&\quad\quad+\|Q^{\frac 12}\|_{\LL^1_2}^8+\| \alpha\|_{\HH^2}^4\big)+e^{(\frac {2(a-\epsilon)}{1+\eta}-6a)s}\|u_0\|^6\big(1+\|\alpha-\frac {F_Q}2\|^2_{L^{\infty}}\big)\bigg)ds\\\nonumber 
&\quad \quad\le e^{-\frac {2(a-\epsilon)}{1+\eta}t}C(\epsilon,\eta,\alpha,Q)\big(1+H(u_0)+\|u_0\|^6\big).
\end{align} 	
Finally, the Gagliardo--Nirenberg  and Young inequalities and the Sobolev embedding theorem imply  the uniform boundedness for the $p$-moment of $\| u\|_{\HH^1}$.		
\end{proof}

Next we show that \eqref{dnls} admits a unique invariant measure $\delta_0$ and a unique stationary solution $0$ in $\HH^1$ similarly  \cite{Dap06}.

\begin{cor}
Under the same condition as Proposition \ref{h1}, the following statements hold:
\begin{enumerate}[(i)]
\item We have
\begin{align*}
\lim_{t\to \infty} P_t\phi(w)=\phi(0),\;\; w\in \HH^1, \;\; \phi \in C_b(\HH^1),
\end{align*}
where $P_t$ is the Markov semigroup associated with the solution $u(t)$.
\item  $\delta_0$ is the unique invariant measure for $P_t$.
\item For any Borel probability measure $\nu \in \mathscr P(\HH^1)$, we have 
  \begin{align*}
\lim_{t\to \infty} \int_{\HH} P_t\phi(w)\nu(dw) =\phi(0).
\end{align*}
\item There exists $b>0$ such that for any functional $\phi \in C_b^1(\HH^1)$, we have 
\begin{align*}
\bigg| P_t\phi(w)-\phi(0)\bigg|\le C\|\phi\|_{C_b^1}e^{-bt}(1+\|w\|_{\HH^1}).
\end{align*}
\end{enumerate}

\end{cor}
\begin{proof}
We show that for any time sequence $\{t_n\}_{n\in \N}$ with $\lim\limits_{n\to \infty} t_n= \infty$, 
the sequence $\{u(t_n)\}_{n\in \N}$ admits a unique limit.
For any $t_n\le t_m$, $n\le m$, by Minkowski and Young inequality and Sobolev embedding theorem, we have 
 \begin{align*}
 &\E [\|u(t_n)-u(t_m)\|_{\HH^1}^2]\\
& \le C \Bigg(
\E\big[\|(S_a (t_m-t_n)-I)u(t_n)\|_{\HH^1}^2\big]
+\E\big[\int_{t_n}^{t_m}\|S_a(t_m-s)\bi\lambda |u(s)|^2u(s)\|_{\HH^1}
 ds\big]^2 \\
 &\quad+\E\big[\int_{t_n}^{t_m}\|S_a(t_m-s)(-\alpha+a)u(s)\|_{\HH^1}ds\big]^2 +\E\big[\|\int_{t_n}^{t_m}S_a(t_m-s)\bi u(s)dW(s)\|_{\HH^1}^2\big]\Bigg)\\
 &\le C\E\big[\|u(t_n)\|_{\HH^1}^2\big]+
 C\E \big[\int_{t_n}^{t_m}e^{-a(t_m-s)}(\|u(s)\|_{\HH^1}+\|u(s)\|_{\HH^1}^3)ds\big]^2\\
 &\quad +C\int_{t_n}^{t_m} \E \big[e^{-2a(t_m-s)}\|u(s)\|_{\HH^1}^2\big]ds,
 \end{align*}
 where $S_a(t):= e^{\bi \Delta t-a t}$. 
The arguments and estimate \eqref{exp-dec} in Proposition  \ref{h1} yield that  for some  $b>0$, we have
 \begin{align*}
  \E [\|u(t_n)-u(t_m)\|_{\HH^1}^2]
&\le C(a,\eta,\alpha,Q)e^{-bt_n}\big(1+H(u_0)+\|u_0\|^6\big).
  \end{align*}
This implies that $\{u(t_n)\}_{n\in \N}$ is a Cauchy sequence and thus $\{u(t)\}_{t\in \R^+}$ admits at least a strong limit. Combining those with the exponential decay estimate 
\eqref{exp-dec}, we get $0$ is the unique strong limit of $\{u(t)\}_{t\in \R^+}$.
The strong mixing property is immediately obtained, and we finish the proof by the exponential decay estimate and strong mixing property.
\end{proof}

To get  the a priori estimates in $\HH^{\bs}$, we introduce the auxiliary Lyapunov functional $f(u):=\|\nabla^\bs u\|^2-\lambda\big\<(-\Delta)^{\bs-1}u,|u|^2u\big\>$ from \cite{CHL16b}.

\begin{prop}
\label{hs}
Assume that $\alpha \in \HH^{\bs} $, $\sup\limits_{x\in\mathbb R}(\frac12 F_Q(x)- \alpha(x))\le -a$, $Q^{\frac 12}\in\LL_2^\bs$ and $u_0\in \mathbb H^\bs$, $\bs\ge2 $.
 For any $p\ge 2$, we have 
\begin{align*}
\sup_{t\in[0,\infty)}\E\Big[\big\|u(t)\big\|_{\mathbb H^\mathbf s}^p\Big]
\le C(\alpha,Q)(1+\|u_0\|_{\HH^{\bs}}^p+\|u_0\|_{\HH^{\bs-1}}^{5p}).
\end{align*}
\end{prop}

\begin{proof}
We prove the uniform boundedness by induction. 
Assume that the $p$-moment of $\|u\|_{\mathbb H^{\bf s-1}}$ is uniformly controlled. 
For simplicity, we show the case $p=2$ under the $\mathbb H^{\bs}$-norm.	
Applying It\^o formula to the functional 
$f(u(t))$, 
we can get the terms similar to those in \cite{CHL16b}.
Similar arguments yield that for $\bs\ge 2$,
\begin{align*}
&\E[(f(u(t)))]\\
&\le f(u_0)-(a-\epsilon)\int_0^t\E \big[\|\nabla^\bs u\|^2\big]ds
+C(\epsilon,\alpha, Q)\int_0^t\E\big[\|u\|_{\HH^{\bs-1}}^4
+\|u\|_{\HH^{\bs-1}}^{10}\big]ds. 
\end{align*}
Since $f(u)\le \| \nabla^{\bs} u\|^2+ C\|u\|_{\HH^{\bs-1}}^4$, iterative arguments similar to those in Proposition \ref{h1} complete the proof.
\end{proof}

\begin{rk}
 Due to the particular structure of charge and energy, 
 the exponential decay estimates in $\HH^\bs$, $\bs\ge 2$ can  also be obtained similarly to Proposition \ref{h1} by iterative arguments. This show that \eqref{dnls} is an ergodic system and admits the unique stationary solution 0
in $\HH^\bs$. This long-time behavior result still holds when 
we consider \eqref{dnls} in a bounded domain with homogeneous boundary condition. 
\end{rk}

Beyond  these a priori estimations,
 we need the  exponential integrability to  construct numerical schemes with strong  and weak convergence order similar to those in \cite{CHL16b,CHLZ17, HJ14}. We also note that this type of exponential integrability has many other applications (see e.g. \cite{CHJ13,  CHL16b, CHLZ17, Dor12, HJ14,HJW13,JP16}).

\begin{prop}\label{exp-u}
Assume   $\alpha\in \HH^2$,   $\sup\limits_{x\in\mathbb R}(\frac12 F_Q(x)- \alpha(x))\le -a$, $Q^\frac12\in \LL_2^2$, and $u_0\in \HH^1$.
There exist $\beta$ and $C$  depending on $\alpha, Q,$ and  $u_0$ such that 
 \begin{align}\label{exp-mom-u}
\sup_{t\in [0,\infty)}\E\left[\exp\bigg( e^{-\beta t}
H(u(t))\bigg)\right]
&\le C.
\end{align}
\end{prop}
\begin{proof}
Denote $\mu(u)=\bi \Delta u+\bi \lambda|u|^2u-\alpha u$ and $\sigma(u)=\bi uQ^{\frac 12}$. 
Simple calculations yield that
\begin{align*}
&DH(u)\mu(u)+\frac 12\text{tr} \big[D^2H(u)\sigma(u)\sigma(u)^*\big]
+\frac1{e^{\beta t}}\|\sigma^*(u)DH(u)\|^2\\
&= \big\<\nabla u,\nabla u (\frac {F_Q }2-\alpha)\big\>
-\sum_k\big\<u,\nabla u (Q^{\frac 12}e_k\nabla Q^{\frac 12}e_k-\nabla \alpha)\big>	\\
&\quad+\sum_k\big\<|\nabla Q^{\frac 12}e_k|^2, |u|^2\big\>+\big\<|u|^4,\alpha-\frac{F_Q }2 \big\>
+\frac 1{2e^{\beta t}}\sum_k\big\< \nabla u, \bi u\nabla Q^{\frac 12}e_k\big\>^2.
\end{align*}
The H\"older, Young and Gagliardo--Nirenberg inequalities, combined with  Corollary \ref{cor-char}, yield that
\begin{align*}
&DH(u)\mu(u)+\frac 12\text{tr} \big[D^2H(u)\sigma(u)\sigma(u)^*\big]
+\frac1{e^{\beta t}}\|\sigma^*(u)DH(u)\|^2\\
&\le -\left(a-\epsilon-\frac 1{2e^{(\beta+2a) t}}
\|u_0\|^2\sum_k\|\nabla Q^{\frac 12}e_k\|_{L^{\infty}}^2\right) \|\nabla u\|^2\\
&\quad+C(\epsilon)\|u_0\|^2e^{-2at}\left(\|Q^{\frac 12}\|_{\LL^2_2}^4
+\|\alpha\|_{\HH^2}^2
+\|u_0\|^4e^{-4at}\|\alpha-\frac {F_Q}2\|^2_{L^{\infty}}\right)
\end{align*} 
Let $\beta\ge -2a$. By the  Gagliardo--Nirenberg  and Young inequalities, we get 
\begin{align*}
&DH(u)\mu(u)+\frac 12tr \big[D^2H(u)\sigma(u)\sigma(u)^*\big]
+\frac1{e^{\beta t}}\|\sigma^*(u)DH(u)\|^2\\
&\le -\left(a-\epsilon-\frac 12\|u_0\|^2\sum_k\|\nabla Q^{\frac 12}e_k\|_{L^{\infty}}^2\right)\frac 2{1+\eta}H(u)
+C(\epsilon,\eta)\|u_0\|^2e^{-2at}\Bigg(\|Q^{\frac 12}\|_{\LL^2_2}^4\\
&\quad 
+\|\alpha\|_{\HH^2}^2+
\|u_0\|^4e^{-4at}\left(\|\alpha-\frac {F_Q}2\|^2_{L^{\infty}}+1\right)\Bigg)\\
&:=-\left(a-\epsilon-\frac 12\|u_0\|^2\sum_k\|\nabla Q^{\frac 12}e_k\|_{L^{\infty}}^2\right)\frac 2{1+\eta}H(u)
+V(\epsilon,\eta,t,u_0).
\end{align*} 
By    \cite[Lemma 3.1] {CHL16b},
we need $\beta\ge \frac {-2a+2\epsilon+\|u_0\|^2\sum\limits_k\|\nabla Q^{\frac 12}e_k\|_{L^{\infty}}^2}{1+\eta}$.
Thus there  always exist $\epsilon$ and $\eta$ such that
$-2a-\beta<0$ and 
\begin{align*}
\sup_{t\in [0,\infty)}\E\left[\exp\left(e^{-\beta t}H(u(t))\right)\right]
\le \E\left[\exp\left(H(u_0)+\int_0^te^{-\beta r}V(\epsilon,\eta,r,u_0)dr\right)\right]\le C.
\end{align*}
\end{proof}

\section{Strong convergence}
We use  a splitting idea similar to that in \cite{CHLZ17,GK03} to discretize \eqref{dnls} and obtain the strong convergence rate independent of the time domain. The key tool is applying the stability in $\HH^2$ and the exponential integrability of both numerical and exact solutions.
The main idea is to split \eqref{dnls} in $T_m=[t_m,t_{m+1})$, $t_m=m\tau$, $m\in \Z_M:=\{0,1,2,\dots, M-1\}$, into a deterministic NLS equation with random initial datum and a linear damped SPDE:
\begin{align}\label{nls-d}
du_\tau^D(t)
&=\left(\bi \Delta u_\tau^D(t)
+\bi \lambda |u_\tau^D(t)|^2 u_\tau^D(t) \right)dt,
\\
\label{nls-s}
du_\tau^S(t)
&=-\alpha u_\tau^S(t)dt+ \bi  u_\tau^S(t) dW(t).
\end{align}
For simplicity, we denote the solution operators of \eqref{nls-d} and \eqref{nls-s} in $T_m$ as $\Phi_{m,t-t_m}^D$ and $\Phi_{m,t-t_m}^S$, respectively.
Next we set the splitting process $u_\tau$ in $T_m$ as 
\begin{align}\label{ut}
u_\tau (t):=u_{\tau,m}^S(t):=(\Phi_{j,t-t_m}^S\Phi_{j,\tau}^D)\prod_{j=1}^{m-1}\big(\Phi_{j,\tau}^S\Phi_{j,\tau}^D\big)u_\tau(0),\quad t\in T_m,
\end{align}
and
\begin{align*}
u_{\tau}^D(t):=u_{\tau,m}^D(t):=\Phi_{j,t-t_m}^D\prod_{j=1}^{m-1}\big(\Phi_{j,\tau}^S\Phi_{j,\tau}^D\big)u_\tau(0),\quad t\in  \{t_m \cup T_m\}/t_{m+1}.
\end{align*}
For the sake of simplicity, we take the initial datum of the splitting process to be $u_\tau(0)=u_0$.
Iterating previous procedures, we obtain a splitting process 
$u_\tau=\{u_\tau(t):\ t\in [0,T]\}$, 
which is left-continuous with finite right-hand limits and $\FFF_t$-adapted.
 We  note that there are some results on numerically approximating SPDEs by splitting schemes (see  \cite{CV10, CHL17c, Dor12,  GK03,  Liu13a} and references therein).
Since  \eqref{nls-d}  has no analytic solution, we apply the Crank--Nicolson type scheme to temporally discretize  \eqref{nls-d}.
Based on the explicitness of the solution of  \eqref{nls-s}, we get the splitting Crank--Nicolson type scheme starting from $u_0$:
\begin{align}\label{splcn}
\begin{cases}
u^D_{m+1}
= u_m+ \bi \tau \Delta  u^D_{m+\frac12}+\bi \lambda \tau \frac {| u_m|^2+| u_{m+1}^D|^2}2  u^D_{m+\frac12},\\
u_{m+1}=\exp\left(-\alpha+\frac {F_Q}2+\bi (W_{t_{m+1}}-W_{t_m})\right) u_{m+1}^D,
\quad m\in \Z_M,
\end{cases}
\end{align}
with $ u_{m+\frac12}^D=\frac12(u_m+ u_{m+1}^D)$.
We can also get the continuous extension of $u_m$ as 
\begin{align*}
\widehat u_{\tau}(t):=\widehat u_{\tau,m}^S(t):=(\Phi_{j,t-t_m}^S\widehat {\Phi_{j,\tau}^D})\prod_{j=1}^{m-1}\big(\Phi_{j,\tau}^S\widehat {\Phi_{j,\tau}^D}\big)u_\tau(0),\quad t\in T_m,  
\end{align*}
where $\widehat {\Phi_{j,\tau}^D}$ is the solution operator of 
the Crank--Nicolson type scheme.

Throughout  this paper, we do not consider the spatial discretization 
since our approach and proof can be extended to the study of a fully discrete scheme as in \cite{CHLZ17}. Some estimates need to be modified accordingly. However, this requires long and technical computations and would probably increase the length of our paper.
For more results on the strong convergence result of spatial approximations  for  the stochastic NLS equation, we refer the reader to \cite{CHL16b,CHLZ17}. However, the study of strong and weak convergence rate of numerical schemes both in time and space for a higher dimensional stochastic NLS equation requires the a priori estimates in a higher Sobolev norm and further investigation.

Next, we always assume that 
$ \sup\limits_{x\in\mathbb R}(\frac12F_Q(x)- \alpha(x))\le -a$.
Since \eqref{nls-d} possesses the charge conservation law and  \eqref{nls-s} is weakly damped, it is not difficult to obtain the 
following  results about the charge of this splitting process. 

\begin{lm}\label{dis-char}
	Let  $\alpha \in \HH^1$,  $\sup\limits_{x\in\mathbb R}(\frac12 F_Q(x)- \alpha(x))\le -a$,  $Q^{\frac 12}\in\LL_2^1$, and $u_0\in \mathbb H$.
	The splitting process $u_\tau=\{u_\tau(t):\ t\in [0,T)\}$ is uniquely solvable and $\FFF_t$-measurable.
	Moreover, for any $t\in [0,T]$ there holds a.s. that 
	\begin{align*}
	\|u_\tau(t)\|^2\le e^{-2at}\|u_0\|^2.
	\end{align*}
	For $t\in T_m$,
	we have 
	\begin{align*}
	\|u_{\tau,m}^S(t)\|^2\le e^{-2at}\|u_0\|^2,\quad 
     \|u_{\tau,m}^D(t)\|^2\le e^{-2at_m}\|u_0\|^2.
	\end{align*}
	
\end{lm}

\begin{prop}
	\label{dis-hs}
	Assume that $\alpha \in \HH^{\bs} $,  
	$\sup\limits_{x\in\mathbb R}(\frac12 F_Q(x)- \alpha(x))\le -a$,
	 $Q^{\frac 12}\in\LL_2^\bs$, and $u_0\in \mathbb H^\bs$, $\bs\ge 1$. Then for any $p\ge 2$, we have 
	\begin{align}\label{dis-hs-ut}
	\sup_{t\in[0,\infty)}\E\Big[\big\|u_\tau(t)\big\|_{\mathbb H^\bs }^p\Big]
	\le C(1+\|u_0\|^p_{\mathbb H^{\bs}}+\|u_0\|_{\mathbb H^{ \bs-1}}^{5p}).
	\end{align}
\end{prop}
\begin{proof}
	For simplicity, we give the proof for $p=2$. 
	The case $p>2$ is  made similar to the proof in \cite[Theorem 2.1]{CHL16b} by applying the It\^o formula to appropriate power of the auxiliary functionals $H$ and $f$, and applying
Burkholder--Davis--Gundy inequality.
	Notice that the energy evolution of splitting process \eqref{ut} is same as Eq. \eqref{dnls} in each interval $T_m$. The It\^o formula, combined with the energy conservation law of Eq. \eqref{nls-d}, yields that 
	\begin{align*}
	&\mathbb E [H(u_{\tau,m}^S(t))]-\mathbb E [H(u_{\tau,m}^D(t_{m}))]\\
	&=\int_{t_m}^t \E \Big[\big\<\nabla u_{\tau,m}^S,\nabla u_{\tau,m}^S (\frac {F_Q}2-\alpha)\big\>\Big]ds
	+\int_{t_m}^t\sum_k\E \Big[\big \<\nabla u_{\tau,m}^S, u_{\tau,m}^S  (Q^{\frac 12}e_k \nabla Q^{\frac 12}e_k- \nabla\alpha)\big\>\Big]ds\\
	&\quad+\int_{t_m}^t \frac 12\sum_k\E \Big[\big\<u_{\tau,m}^S, u_{\tau,m}^S |\nabla Q^{\frac 12}e_k|^2 \big\>\Big] ds
	+\int_{t_m}^t \lambda\E \Big[\big\<|u_{\tau,m}^S|^2u_{\tau,m}^S, u_{\tau,m}^S(\alpha-\frac {F_Q}2) \big\>\Big]ds.
	\end{align*}
	Similar to Proposition \ref{h1}, we get 
	\begin{align*}	
	\E [H(u_{\tau,m}^S(t))]
	&\le  H(u_{\tau,m}^S(t_m))-\frac {2(a-\epsilon)}{1+\eta}\int_{t_m}^t\E\big[H(u_{\tau,m}^S)\big]ds
	+C(\epsilon,\eta)\int_{t_m}^t\bigg(\|u_{\tau,m}^S\|^2
	\big(\|Q^{\frac 12}\|_{\LL^1_2}^4\\
	&\quad+
\|Q^{\frac 12}\|_{\LL^1_2}^8+\|\nabla \alpha\|^4\big)+
	\|u_{\tau,m}^S\|^6\|\alpha-\frac {F_Q}2\|^2_{L^{\infty}}+\|u_{\tau,m}^S\|^6\bigg)ds.
	\end{align*} 
The Gronwall inequality implies 
\begin{align*}
\E[H(u_{\tau,m}^S(t))]&\le  e^{-\frac {2(a-\epsilon)}{1+\eta}(t-t_m)}H(u_{\tau,m}^S(t_m))+e^{-\frac{2(a-\epsilon)}{1+\eta}t }C(\epsilon,\eta, \alpha,Q,\|u_0\|)(t-t_m).
\end{align*}	
Then by repeating the above procedures in each interval and combining them with  discrete Gronwall inequality, we obtain
\begin{align*}
\E [H(u_\tau(t))]\le e^{-\frac {2(a-\epsilon)}{1+\eta}t}H(u_0)+
e^{-\frac{2(a-\epsilon)}{1+\eta}t } (1+t)C(\epsilon,\eta,\alpha,Q)
\le H(u_0)+C(\epsilon,\eta,\alpha,Q).
\end{align*}	
Then 
similar arguments lead to the uniform boundedness for $p\ge 2$. \\
Next, we turn to estimate $\E [\|u\|_{\HH^\bs }^2], \bs \ge 2$.
Similar to Proposition \ref{hs}, we have
\begin{align*}
&f(u_{\tau,m}^D(t))-f(u_{\tau,m}^D(t_m)) \\
&=-\int_{t_m}^{t} \left\<(-\Delta)^{\bs-1} u_{\tau,m}^D,\bi |u_{\tau,m}^D|^4u_{\tau,m}^D\right\>dr
-\lambda \int_0^t \Big\<(-\Delta)^{\bs-1} u_{\tau,m}^D, 
3\bi |u_{\tau,m}^D|^2\Delta u_{\tau,m}^D \Big\> dr\\
&\quad-\lambda \int_{t_m}^{t}  \left\<(-\Delta)^{\bs-1}  u_{\tau,m}^D,
4\bi |\nabla  u_{\tau,m}^D|^2 u_\tau^D
+2\bi (\nabla  u_{\tau,m}^D)^2  \overline{u_{\tau,m}^D}\right\>dr \\
&\le \epsilon \int_{t_m}^t f(u_{\tau,m}^D)ds
+C(\epsilon,\alpha,Q)\int_{t_m}^t\left(\|u_{\tau,m}^D\|_{\HH^{\bs-1}}^4
+\|u_{\tau,m}^D\|_{\HH^{\bs- 1}}^{10}\right)ds.
\end{align*}
By the  Gronwall inequality, we obtain 
\begin{align*}
f(u_{\tau,m}^D(t_{m+1}))\le e^{\epsilon \tau}f(u_{\tau,m}^D(t_m)) 
+C(\epsilon,\alpha,Q)\int_{t_m}^{t_{m+1}}\left(\|u_{\tau,m}^D\|_{\HH^{\bs-1}}^4
+\|u_{\tau,m}^D\|_{\HH^{\bs-1}}^{10}\right)ds.
\end{align*}
On the other hand, the It\^o formula and the  Young and  Gagliardo--Nirenberg inequalities yield that 
\begin{align*}
\E[f(u^S_{\tau,m}(t))]&\le \E [f(u^D_{\tau,m}(t_{m+1}))]-(a-\epsilon)\int_{t_m}^t\E \big[\|\nabla^\bs u_{\tau,m}^S\|^2\big]ds\\
&\quad+C(\epsilon,\alpha,Q)\int_{t_m}^t\E\big[\|u_{\tau,m}^S\|_{\HH^{\bs-1}}^4
+\|u_{\tau,m}^S\|_{\HH^{\bs-1}}^{10}\big]ds. 
\end{align*}
Again by the Gronwall inequality, 
we get
\begin{align*}
\E[f(u^S_{\tau,m}(t))]
&\le
e^{-(a-\epsilon)(t-t_m)+\epsilon \tau} \E [f(u^D_{\tau,m}(t_{m}))]\\
&\quad+C(\epsilon,\alpha,Q)\int_{t_m}^t e^{-(a-\epsilon)(t-s)} \E\big[\|u_{\tau,m}^S\|_{\HH^{\bs-1}}^4
+\|u_{\tau,m}^S\|_{\HH^{\bs-1}}^{10}\big]ds\\
&\quad+C(\epsilon,\alpha,Q)e^{-(a-\epsilon)(t-t_m)} \int_{t_m}^{t_{m+1}}\E\left[\|u_{\tau,m}^D\|_{\HH^{\bs-1}}^4
+\|u_{\tau,m}^D\|_{\HH^{\bs-1}}^{10}\right]ds.
\end{align*}
Finally, the discrete Gronwall inequality, together with the induction hypothesis,  leads to 
\begin{align*}
\E[(f(u_\tau(t)))]&\le  Ce^{-(a-2\epsilon)t}f(u_0)+\frac {1-e^{-(a-2\epsilon)T}}{1-e^{-(a-2\epsilon)\tau}}
C(\epsilon, \alpha,Q,u_0)\tau\\
&\le f(u_0)+C(\epsilon, \alpha,Q,u_0),
\end{align*}
where we use the fact that 
$\frac \tau{1-e^{-c\tau}}\le \frac {1+c\tau}c$.
The relationship $ \|\nabla^\bs u\|^2-C\|u\|_{\HH^{\bs-1}}^4\le f(u)\le \|\nabla^\bs u\|^2+C\|u\|_{\HH^{\bs-1}}^4$ and induction arguments finish the proof.
\end{proof}

We also need a priori  estimation on numerical solution of the splitting Crank-Nicolson scheme \eqref{splcn}. The detail proof for the following lemma  is omitted since it is similar to the proof of Proposition \ref{dis-hs}.

\begin{lm}\label{dis-char-cn}
	Let  $\alpha \in \HH^1 $, $\sup\limits_{x\in\mathbb R}(\frac12 F_Q(x)$ $- \alpha(x))\le -a$,
	$Q^{\frac 12}\in\LL_2^1,$ and $u_0\in \mathbb H^1$.
	The splitting process $u_m, m\in \mathbb Z_M$ is uniquely solvable and $\FFF_{t_m}$-measurable.
	Moreover, it holds a.s. that 
	\begin{align}\label{cha-spl-cn}
	\|u_{m}\|^2\le e^{-2at_m}\|u_0\|^2.
	\end{align}
	For $t\in T_m$, the energy of $u_m$ is uniformly bounded. More precisely, for any $p\ge1$,
	there exists $b>0$ such that
	\begin{align*}
	\sup_{m\in \Z_M}\E[H^p(u_m)]\le Ce^{-bt_m}(1+H^p(u_0)).
	\end{align*}

\end{lm}

\begin{prop}\label{h2-um}
	     Assume  that $\alpha \in \HH^{2}$, $\sup\limits_{x\in\mathbb R}(\frac12 F_Q(x)$ $- \alpha(x))\le -a$, $Q^{\frac 12}\in\LL_2^2$, and $u_0\in \mathbb H^2$.
		Then for any $p\ge 2$, there exists a constant $C=C( \alpha,Q, u_0,p)$ such that  
		\begin{align}\label{h2-um0}
		\sup_{m\in \Z_{M}}\E\left[\|u_m\|_{\HH^2}^p \right]\le C.
		\end{align}
\end{prop}
\begin{proof}
	Arguments similar to \cite[Lemma 3.3]{CHLZ17}, combined with the Young inequality, yield that 
	\begin{align*}
	f( u^{D}_{m+1})
&	\le f(u_m) +\frac {\epsilon \tau}2 (\|\Delta u_m\|^2 +\|\Delta u^{D}_{m+1}\|^2)\\
&\quad	+C(\epsilon) \tau \left(1+\|\nabla u^{D}_{m+1}\|^{12}+\|\nabla u_m\|^{12}\right).
	\end{align*}
Then we have 
	\begin{align*}
	f( u^{D}_{m+1})&\le  \frac{1+\frac{\epsilon \tau} 2}{1-\frac{\epsilon \tau} 2}f(  u_m)
	+\frac{C(\epsilon)\tau}{1-\frac{\epsilon \tau} 2}\left(1+\|\nabla u^{D}_{m+1}\|^{12}+\|\nabla u_m\|^{12}\right).
	\end{align*}
Let $\epsilon\tau\le1$. We get 
\begin{align*}
f( u^{D}_{m+1})
&\le (1+2\epsilon\tau) f(  u_m)
+C(\epsilon) \tau \left(1+\|\nabla u^{D}_{m+1}\|^{12}+\|\nabla u_m\|^{12}\right).
\end{align*}	
Notice that $u_m$ can be extended to a continuous process 
$\widehat  u^S_{\tau,m}(t)$ with  $\widehat  u^S_{\tau, m}(t_m)=u^{D}_{m+1}$ in $ T_m$.
The  arguments in Proposition \ref{dis-hs}, together with Lemma \ref{dis-char-cn} show that for some $b_1>0$,
\begin{align*}
\E[f(\widehat  u^S_{\tau,m}(t))]
&\le
e^{-(a-\epsilon)(t-t_m)} \E [f(\widehat  u^D_{m+1})]
+e^{-b_1t}C(\epsilon, \alpha,Q,u_0)\tau\\
&\le e^{-(a-\epsilon)(t-t_m)}(1+2\epsilon\tau) \E[f(u_m)]+e^{-\min(b_1,a-3\epsilon)t}C(\epsilon,\alpha, Q,u_0)\tau.
\end{align*}
Using the discrete Gronwall inequality, we obtain 
\begin{align*}
\E[f(\widehat  u_\tau(t))]
\le Ce^{-(a-3\epsilon)t}f(u_0)+e^{-\min(b_1,a-3\epsilon)t}
(t+1)C(\epsilon,\alpha, Q,u_0)
\le
f(u_0)+C(\epsilon, \alpha,Q,u_0),
\end{align*}
which yields the uniform boundedness of $f(u_{m}),m\in \mathbb Z_M$,
 and thus $\|u_{m}\|_{\HH^{2}},m\in \mathbb Z_M$.
 The proof of the case $p>2$ is similar.
 
\end{proof}

To analyze the strong and weak order of the proposed scheme, we need to 
show some exponential integrability of $u_m$ and $u_\tau$ based on  \cite[Lemma 3.1] {CHL16b}. These 
 exponential integrability properties can be used to deduce the continuous dependence on initial data of $u_m$ and $u_\tau$
 as in
\cite{CHL16b, HJ14}.  

\begin{prop}\label{exp-ut}
	Let  $\alpha \in \HH^1 $, $\sup\limits_{x\in\mathbb R}(\frac12 F_Q(x)- \alpha(x))\le -a$,
	$Q^\frac12\in \LL_2^1$ and $u_0\in \HH^1$.
	There exist $\beta$ and $C=C(\alpha,Q,u_0)$ such that 
	\begin{align}\label{exp-mom-ut}
	\E\left[\exp\bigg( e^{-\beta t}
	H(u_\tau(t))\bigg)\right]
	&\le C,\\\label{exp-mom-um}
		\E \left[\exp\bigg( e^{-\beta t_m}
		H(u_m)\bigg)\right]
		&\le C.
	\end{align}
\end{prop}

\begin{proof}
We first prove the estimation \eqref{exp-mom-ut}.
Since  \eqref{nls-s} has the same energy evolution as 
 \eqref{dnls} and  \eqref{nls-d} possesses the energy conservation law, by Proposition \ref{exp-u} we have in $T_m$ that 
there  always exists 
$\beta >-2a+\|u_0\|^2\sum\limits_k\|\nabla Q^{\frac 12}e_k\|_{L^{\infty}}^2$ 
such that 
\begin{align*}
&\E\left[\exp\left(e^{-\beta t}H(u_\tau(t))\right)\right]\\
&\le \E\left[\exp\left(e^{-\beta t_m}H(u_{\tau,m}^S(t_m))
+\int_{t_m}^te^{-\beta s}V(\epsilon,\eta,s,u_0)ds\right)\right]\\
&\le \E\left[\exp\left(e^{-\beta t_m}H(u_{\tau,m}^D(t_{{m}}))
+\int_{t_m}^te^{-\beta s}V(\epsilon,\eta,s,u_0)ds\right)\right],
\end{align*}
where $V(\epsilon,\eta,s,u_0)$ is the function appearing   in the proof of Proposition \ref{exp-u}.

Repeating the above procedures in each interval, we deduce that 
\begin{align*}
\E\left[\exp\left(e^{-\beta t}H(u_\tau(t))\right)\right]
\le \E\left[\exp\left(H(u_0)
+\int_{0}^te^{-\beta s}V(\epsilon,\eta,s,u_0)ds\right)\right]\le C(\epsilon,\eta,\alpha, Q,u_0),
\end{align*}
which verifies  estimation \eqref{exp-mom-ut}.
Similar arguments yield estimation \eqref{exp-mom-um}.
\end{proof}

\begin{rk}\label{exp-utd}
	Under the  condition of Proposition \ref{exp-ut}, by the same procedures we can obtain that 
	\begin{align}
	\E\left[\exp\bigg( e^{-\beta t}
	H(u_\tau^D(t))\bigg)\right]\le C.
	\end{align}
	\end{rk}

\begin{cor}\label{u-ut-exp}
	Under the  condition of Proposition \ref{exp-ut},
	 there exists a constant $C=C(\alpha,Q,u_0)$ for any $p\ge 1$ such that 
	\begin{align}\label{u-ut-exp0}
	\left\|\exp\left(2\int_0^T  \|u(s)\|_{L^\infty}
	\|u_\tau^D(s)\|_{L^\infty} ds\right)\right\|_{L^p(\Omega)}\le C
	\end{align}
	and 
	\begin{align}\label{ut-um-exp}
	\left\|\exp\left(2\sum_{m\in \mathbb Z_M}  \|u_\tau(t_m) \|_{L^\infty}
	\|u_m\|_{L^\infty} \tau \right)\right\|_{L^p(\Omega)}\le C.
	\end{align}
\end{cor}

\begin{proof}
	By the Cauchy--Schwarz, Gagliardo--Nirenberg and Young inequalities, for $0<\eta<1$ we have
	\begin{align*}
	&\left\|\exp\left(2\int_0^T  \|u(s)\|_{L^\infty}
	\|u_\tau^D(s)\|_{L^\infty} ds\right)\right\|_{L^p(\Omega)} \\
	&\le \left\|\exp\left(\int_0^T  2e^{-at}\|u_0\|  \|\nabla u\| ds\right)\right\|_{L^{2p}(\Omega)}
	\left\|\exp\left(\int_0^T 2e^{-at}\|u_0\|  \|\nabla u^D_\tau\| ds\right)\right\|_{L^{2p}(\Omega)} \\
	&\le  \sqrt[2p]{\E\Bigg[\exp\left(\int_0^T \frac {4p\sqrt 2}{\sqrt {1-\eta}}e^{-(a-\frac \beta 2)t}\|u_0\| 
		e^{-\frac \beta 2 t}\sqrt{\frac {1-\eta}{2}}\|\nabla u\| ds\right)\Bigg] }\\
	 &\quad \cdot \sqrt[2p]{\E\Bigg[\exp\left(\int_0^T \frac {4p\sqrt 2}{\sqrt {1-\eta}}e^{-(a-\frac \beta 2)t}\|u_0\| 
	 	e^{-\frac \beta 2 t}\sqrt{\frac {1-\eta}{2}}\|\nabla u\| ds\right)\Bigg]},
	\end{align*}
	where $\beta <2a$ is as presented in Proposition \ref{exp-ut}.
	Then the Jensen, Minkovski and  H\"older  inequalities yield that 
	\begin{align*}
		&\left\|\exp\left(2\int_0^T  \|u(s)\|_{L^\infty}
		\|u_\tau^D(s)\|_{L^\infty} ds\right)\right\|_{L^p(\Omega)} \\
		&\le \sqrt[2p]{\sup_{t\in[0,T]}\E\Bigg[\exp\left(\frac {4p\sqrt 2(1-e^{-(a-\frac\beta 2)T})}{\sqrt {(1-\eta)(a-\frac \beta 2)}}\|u_0\| 
			e^{-\frac \beta 2 t}\sqrt{\frac {1-\eta}{2}}\|\nabla u\| \right)\Bigg] }\\
		&\quad \cdot \sqrt[2p]{\sup_{t\in[0,T]}\E \Bigg[\exp\left(\frac {4p\sqrt 2(1-e^{-(a-\frac \beta 2)T})}{\sqrt {(1-\eta)(a-\frac \beta 2)}}\|u_0\| 
			e^{-\frac \beta 2 t}\sqrt{\frac {1-\eta}{2}}\|\nabla u_\tau^D\| \right)\Bigg]}\\
		&\le C(a,\beta,\eta,\|u_0\|)\sqrt[2p]{\sup_{t\in[0,T]}\E \Bigg[\exp\left(\frac {(1-\eta)e^{-\beta t}}{2}\|\nabla u(t)\|^2-\frac {e^{-\beta t}}{8\eta}\|u(t)\|^6 \right)}\Bigg]\\
		&\quad \cdot \sqrt[2p]{\sup_{t\in[0,T]}\E \Bigg[\exp\left(\frac {(1-\eta)e^{-\beta t}}{2}\|\nabla u_\tau^D(t)\|^2 -\frac {e^{-\beta t}}{8\eta}\|u(t)\|^6\right)}\Bigg]\\
		&\le C(a,\beta,\eta,\|u_0\|)\sqrt[2p]{\sup_{t\in[0,T]}\E \Bigg[\exp\left(e^{-\beta t}H(u(t)) \right)}\Bigg]\cdot \sqrt[2p]{\sup_{t\in[0,T]}\E \Bigg[\exp\left(e^{-\beta t}H(u_\tau^D(t))\right)}\Bigg].
	\end{align*}
From the above estimations, Propositions \ref{exp-u} and  \ref{exp-ut}  and Remark \ref{exp-utd} yield  \eqref{u-ut-exp0}.
Next, we turn to the discrete case  \eqref{ut-um-exp}. Similarly, the H\"older, Gagliardo--Nirenberg and Jensen inequalities yield that 
\begin{align*}		
&\left\|\exp\left(2\tau \sum_{m\in \mathbb Z_M}  \|u_\tau(t_m) \|_{L^\infty}
\|u_m\|_{L^\infty} \right)\right\|_{L^p(\Omega)}\\
&\le \sqrt[2p]{\E\Bigg[\exp\Big(4p\tau \sum_{m\in \Z_M}e^{-(a-\frac \beta 2)t_m}\|u_0\|e^{-\frac \beta 2t_m} \|\nabla u_\tau(t_m)\| \Big)\Bigg]}
\\
&\quad \cdot \sqrt[2p]{\E\Bigg[\exp\Big(4p\tau \sum_{m\in \Z_M}e^{-(a-\frac \beta 2)t_m}\|u_0\| e^{-\frac \beta 2t_m}\|\nabla u_m\| \Big)\Bigg]}\\
&\le \sqrt[2p]{\sup_{m\in \Z_M}\E\Bigg[\exp\Big( 4\sqrt 2p\frac {1+(a-\frac \beta 2)\tau}{\sqrt{1-\eta}(a-\frac \beta 2)}\|u_0\| e^{-\frac \beta 2t_m}
	\sqrt{\frac {1-\eta}{2}}\|\nabla u_\tau(t_m)\| \Big) \Bigg]}\\
&\quad \cdot
\sqrt[2p]{\sup_{m\in \Z_M} \E\Bigg[\exp\Big( 4p\sqrt 2\frac{1+(a-\frac \beta 2)\tau}{\sqrt{1-\eta}(a-\frac \beta 2)}\|u_0\|e^{-\frac \beta 2t_m}\sqrt{\frac {1-\eta}{2}}\|\nabla u_m\| \Big) \Bigg]}\\
&\le C(a,\beta,\eta,u_0)\sqrt[2p]{\sup_{m\in \Z_M}\E\Bigg[\exp\Big( e^{- \beta t_m}
	H(u_\tau(t_m)) \Big) \Bigg]}\cdot
\sqrt[2p]{\sup_{m\in \Z_M}\E\Bigg[\exp\Big( e^{- \beta t_m}
	H(u_m) \Big) \Bigg]}.
\end{align*}
\end{proof}

Based on these a priori estimations and exponential integrability, 
we can deduce the strong convergence rate for the splitting Crank--Nicolson type scheme. 
We remark that when the damped assumption $\sup\limits_{x\in\mathbb R}(\frac12 F_Q(x)$ $- \alpha(x))\le -a$ does not hold, the strong convergence rate of the proposed scheme can also be obtained. However, we cannot expect that the constant $C$ in the upper estimate of the  strong convergence rate to be independent of time since the a priori estimate depends on the time interval. A similar situation occurs when we 
study the weak order of the proposed scheme.

\begin{tm}
Let  $\alpha \in \HH^{2} $,  $\sup\limits_{x\in\mathbb R}(\frac12 F_Q(x)$ $- \alpha(x))\le -a$, and $Q^\frac12\in \LL_2^2$.
Then for $p\ge 1$, there exists a constant $C=C(\alpha,Q,u_0,p )$ such that
\begin{align*} 
\mathbb E\Big[\sup_{m\in \Z_{M}}\|u(t_m)-u^m\|^p\Big]\le 
C\tau^{\frac p2}.
\end{align*}
\end{tm}

\begin{proof}
For simplicity, we give the proof for $p=2$. The proof of  case $p>2$ can be similarly obtained by  using a priori estimates in higher p-moments of numerical and exact solutions in Sobolev norms.
Similar to the proof in \cite{CHLZ17}, we split the error $\mathbb E\Big[\sup\limits_{m\in \Z_{M}}\|u(t_m)-u^m\|^2\Big]$ as follows:
\begin{align*}
&\E \left[\sup_{m\in \Z_{M}} \|u(t_{m})-u_{m}\|^{2}\right] \\
&\le 2 \E \left[\sup_{m\in \Z_{M}} \|u(t_{m})-u_\tau(t_{m})\|^{2}\right]
+2\E \left[\sup_{m\in \Z_{M}} \|u_\tau(t_{m})-u_{m}\|^{2}\right].
\end{align*}
Denote $e_m:=u(t_{m})-u_\tau(t_{m})$, $\widehat{e_m}:=u_\tau(t_{m})-u_{m}$. We first estimate the first term $\E \left[\sup\limits_{m\in \Z_{M}} \|e_m\|^{2}\right]$. By the It\^o formula, the definition of $u_\tau$, the Gagliardo--Nirenberg inequality, and  arguments  similar to \cite[Theorem 2.2]{CHLZ17}, we get
\begin{align*}
\|e_{m+1}\|^2
&\le \left\|e_m-\int_{t_m}^{t_{m+1}}
\bi \left[\Delta u_{\tau,m}^D+\lambda |u_{\tau,m}^D|^2 u_{\tau,m}^D\right] dr\right\|^2-a\int_{t_m}^{t_{m+1}}\left\| u-u_{\tau,m}^{S}\right\|^2ds \\
&\quad +2\int_{t_m}^{t_{m+1}} \left\<u-u_{\tau,m}^S,
\bi \left[\Delta u+\lambda |u|^2 u \right] \right\>ds\\
&\le (1-(a+\epsilon)\tau)\|e_m\|^2 +2\left\< e_m,\int_{t_m}^{t_{m+1}}  
\bi \left[\Delta u -\Delta u_{\tau,m}^D+\lambda|u|^2 u-\lambda |u_{\tau,m}^D|^2 u_{\tau,m}^D\right] dr \right\>\\
&\quad+C(\epsilon,u_0)\tau\int_{t_m}^{t_{m+1}} \left[1+\left\|u_{\tau,m}^D\right\|_{\HH^2}^2+\left\|u_{\tau,m}^S\right\|_{\HH^2}^2+\left\|u_\tau\right\|_{\HH^2}^2+\left\|u\right\|_{\HH^2}^2\right]ds\\
&\quad+C(\epsilon,u_0)\Bigg(\int_{t_m}^{t_{m+1}}\|W(s)-W(t_m)\|_{\HH^1}^2(1+\|u\|_{\HH^2}^2)ds\\
&\quad +\int_{t_m}^{t_{m+1}}\left\|\int_{t_m}^{s}\left(u(r)-u_{\tau,m}^S(r)\right)dW(r)\right\|^2ds+R^m_1 
+R^m_2+R^m_3\Bigg),
\end{align*}
where 
\begin{align*}
R^m_1&:=\int_{t_m}^{t_{m+1}}
\left\|\int_{t_m}^s  \int_{t_m}^r \left[\bi \Delta u
+\bi \lambda |u|^2 u 
-\alpha (u-u_{\tau,m}^S) \right] dr_1dW(r)\right\|\Big(1+\|u\|_{\HH^2}\Big) ds,\\
R^m_2&:=\int_{t_m}^{t_{m+1}}\left\|\int_{t_m}^s \int_{t_m}^{r}
\left[u(r_1)-u_{\tau,m}^S(r_1)\right]dW(r_1)dW(r)\right\|\Big(1+\|u\|_{\HH^2}\Big)ds,\\
R^m_3&:=\int_{t_m}^{t_{m+1}}\left\|\int_{t_m}^s \Big(\int_{t_m}^{r}
\bi\left[\Delta u+\lambda|u|^2u\right]dr_1+\int_{t_m}^{t_{m+1}}
\bi\left[\Delta u^D_\tau+\lambda|u^D_\tau|^2u^D_\tau\right]dr_1\Big)dW(r)\right\|\Big(1+\|u\|_{\HH^2}\Big)ds
\end{align*}	
Integrating by parts, we get 
\begin{align*}
&2\left\< e_m,\int_{t_m}^{t_{m+1}}  
\bi \left[\Delta u -\Delta u_\tau^D+\lambda|u|^2 u-\lambda |u_{\tau,m}^D|^2 u_{\tau,m}^D\right] ds \right\>\\
&=2\int_{t_m}^{t_{m+1}}  
\left\<\Delta e_m, \bi \left[u-u_{\tau,m}^D \right] \right\> ds
+2\lambda \int_{t_m}^{t_{m+1}}  \left\< e_m, 
\bi \left[|u|^2 u-|u_{\tau,m}^D|^2 u_{\tau,m}^D \right]\right\>  ds,
\end{align*}
The H\"older inequality,  cubic difference formula $|a|^2a-|b|^2b=(|a|^2+|b|^2)(a-b)+ab(\overline a-\overline b)$,  
the Cauchy-Schwarz and Gagliardo--Nirenberg inequalities imply  
\begin{align*}
&\left\< e_m,\int_{t_m}^{t_{m+1}}  
\bi \left[\Delta u -\Delta u_{\tau,m}^D+\lambda|u|^2 u-\lambda |u_{\tau,m}^D|^2 u_{\tau,m}^D\right] ds \right\>\\
&\le \Big(\frac \epsilon 2\tau +\int_{t_m}^{t_{m+1}} 
\|u\|_{L_{\infty}}\|u_{\tau,m}^D\|_{L_{\infty}} ds\Big)\|e_m\|^2+C(\epsilon)\int_{t_m}^{t_{m+1}} 
\left\| \int_{t_m}^s u(r)dW(r)\right\|^2_{\HH^2} ds\\
&\quad+C(\epsilon,u_0) \int_{t_m}^{t_{m+1}}\Bigg( \|u(s)\|_{\HH^1}^2
+\|u_{\tau,m}^D(s)\|_{\HH^1}^2 \Bigg) \Bigg(\tau \int_{t_m}^{s} \Big(1+\|u(r)\|_{\HH^2}^2
+\|u_{\tau,m}^D(t)\|_{\HH^2}^2 \Big)dr \\
&\quad+ \left\|\int_{t_m}^su(r)dW(r)\right\|_{\HH^2}^2
\Bigg) ds+
C(\epsilon,u_0)\tau\int_{t_m}^{t_{m+1}} \left[1+\left\|u_{\tau,m}^D\right\|_{\HH^2}^2+\left\|u_{\tau,m}\right\|_{\HH^2}^2+\left\|u\right\|_{\HH^2}^2\right]ds.
\end{align*}
Thus we conclude that 
\begin{align*}
\|e_{m+1}\|^2
&\le \|e_m\|^2+\left(-(a-2\epsilon)\tau+2\int_{t_m}^{t_{m+1}} \|u(s)\|_{L_{\infty}}\|u_{\tau,m}^D(s)\|_{L_{\infty}}ds \right) \|e_m\|^2   \\
&\quad+C\Bigg(\tau\int_{t_m}^{t_{m+1}} \left[1+\left\|u_{\tau,m}^D\right\|_{\HH^2}^4+\left\|u_{\tau,m}\right\|_{\HH^2}^4
+\left\|u\right\|_{\HH^2}^4\right]ds\Bigg)\\
&\quad+R_1^m+R_2^m+R_3^m+
C\Bigg(\int_{t_m}^{t_{m+1}}\|W(s)-W(t_m)\|_{\HH^1}^2\Big(1+\|u\|_{\HH^2}^2\Big)ds\Bigg)\\
&\quad+
C\Bigg(\int_{t_m}^{t_{m+1}}\left\|\int_{t_m}^{s}\left(u(r)-u_{\tau,m}^S(r)\right)dW(r)\right\|^2ds\Bigg)\\
&\quad+
C\Bigg(\int_{t_m}^{t_{m+1}}\Big(1+ \|u(s)\|_{\HH^1}^2
+\|u_{\tau,m}^D(s)\|_{\HH^1}^2 \Big)\left\|\int_{t_m}^su(r)dW(r)\right\|_{\HH^2}^2ds\Bigg)\\
&\quad=:\|e_m\|^2+\left(-(a-2\epsilon)\tau+2\int_{t_m}^{t_{m+1}} \|u(s)\|_{L_{\infty}}\|u_{\tau,m}^D(s)\|_{L_{\infty}}ds \right) \|e_m\|^2\\ 
&\quad+R_0^m+R_1^m+R_2^m+R_3^m+R_4^m+R_5^m+R_6^m.
\end{align*}
Then repeating the above procedures yields that 
\begin{align*}
&\|e_{m+1}\|^2\\
&\le \exp\Bigg(-(a-2\epsilon)(m+1)\tau +\int_{0}^{t_{m+1}}\|u(s)\|_{L_{\infty}}\|u_\tau^D(s)\|_{L_{\infty}}ds\Bigg)\|e_0\|^2\\
&\quad+\sum_{k=1}^{m+1}\sum_{i=0}^6\exp\Bigg(-(a-2\epsilon)(m+1-k)\tau +\int_{t_k}^{t_{m+1}}\|u(s)\|_{L_{\infty}}\|u_{\tau,m}^D(s)\|_{L_{\infty}}ds\Bigg)R_i^{k-1}.
\end{align*}
where terms $R_i^j$, $i=0,\dots,6$, $j=0,\dots,M-1$, can be
controlled by Propositions \ref{hs}, \ref{dis-hs}, and  \ref{h2-um}. We omit the detailed computations which are similar
to \cite[Lemma 2.4]{CHLZ17} and obtain $\|R_i^j\|_{L^2(\Omega)}\le C\tau^2$. 
The exponential moment is bounded by the estimation \eqref{u-ut-exp0} of Corollary \ref{u-ut-exp}.
Thus we get  for $\tau<1$,
\begin{align*}
\E[\|e_{m+1}\|^2]
&\le C\tau^2\frac {1- \exp\big(-(a-2\epsilon)T\big)}{1-\exp\big(-(a-2\epsilon)\tau\big)}\le C\tau.
\end{align*}
In fact, we can obtain the stronger result, 
\begin{align*}
\sup_{m\in Z_{M-1}}\E\left[\|e_{m+1}\|^2\right]
&\le \sum_{k=1}^{M}\sum_{i=0}^6 \Big\|R_i^{k-1}\Big\|_{L^2(\Omega)}\Bigg\|\exp\Bigg(-(a-2\epsilon)(m+1-k)\tau \\
&\quad+\int_{t_k}^{t_{M}}\|u(s)\|_{L_{\infty}}\|u_{\tau,m}^D(s)\|_{L_{\infty}}ds\Bigg)\Bigg\|_{L^2(\Omega)}
\le C\tau,
\end{align*}

Next, we estimate the term $\E \left[\sup\limits_{m\in \Z_{M}} \|\widehat {e_m}\|^{2}\right]$.
Similar  to the previous arguments, we get 
\begin{align*}
&\|\widehat {e_{m+1}}\|^2\\ 
&\le\|\widehat {e_m}\|^2+\left(-(a-2\epsilon)\tau+2\tau\left\|u_\tau(t_n)\right\|_{L_{\infty}}\left\|u_n\right\|_{L_\infty} \right) \|\widehat {e_m}\|^2+\Bigg(\tau^2\Big(1+\|u^{D}_{m+1} \|_{\HH^2}^6\\
&\quad +\|u_{m} \|_{\HH^2}^6\Big)+\tau \int_{t_m}^{t_{m+1}}
\Big(\|u_{\tau,m}^S(r)\|_{\HH^2}^2+\|\widehat {u_{\tau,m}^S(r)}\|_{\HH^2}^2
+\| u^{D}_{\tau,m}(r)\|_{\HH^2}^4\Big)dr
\Bigg)\\
&:=\|\widehat {e_m}\|^2+\left(-(a-2\epsilon)\tau+2\tau\left\|u_\tau(t_m)\right\|_{L_{\infty}}\left\|u_m\right\|_{L_\infty} \right) \|\widehat {e_m}\|^2
+\widehat{R^m}
\end{align*}
Then taking expectations on both sides and using the H\"older inequality yields that 
\begin{align*}
&\E \left[\sup_{m\in \Z_{M-1}}\|\widehat {e_{m+1}}\|^2\right]  \\
&\le
\sum_{k=0}^{M-1} \Bigg\|\exp\Bigg(-(a-2\epsilon)(m+1-k)\tau \\
&\quad+\int_{t_k}^{t_{m+1}}\|u(s)\|_{L_{\infty}}\|u_{\tau,m}^D(s)\|_{L_{\infty}}ds\Bigg)\Bigg\|_{L^2(\Omega)}
\Big\|\widehat {R^{k}}\Big\|_{L^2(\Omega)}.
\end{align*}
Then the estimation \eqref{ut-um-exp} in Corollary \ref{u-ut-exp}, combined with a priori estimations in  
Propositions \ref{hs}, \ref{dis-hs} and  \ref{h2-um}, implies that 
\begin{align*}
&\E \left[\sup_{m\in \Z_{M}}\|\widehat {e_{m}}\|^2\right] \le C\tau.
\end{align*}
From the estimations about $e_{m}$ and $\widehat {e_{m}}$, we obtain the strong error estimate
\begin{align*}
\mathbb E\Big[\sup_{m\in \Z_{M}}\|u(t_m)-u^m\|^2\Big]\le 
C\tau.
\end{align*}
\end{proof}

\section{Weak convergence}
In this section, we first study the regularity of the Kolmogrov equation of  \eqref{dnls}.
 With the help of this Kolmogrov equation, we transform the weak error into two parts, 
one is from  the splitting approach and the other is from the deterministic Crank--Nicolson type discretization.
As a consequence, the rate of weak convergence  is shown to be twice that of strong convergence. 
This is the first result about the weak order of numerical schemes approximating the stochastic nonlinear Schr\"odinger equation with nonmonotone coefficients.

It is well known  that $U(t,u_0):=\E\left[\phi(u(t,u_0)) \right]$  satisfies the following infinite-dimensional Kolmogorov
equation  (see e.g. \cite{BD06}):
\begin{equation*}
\left\{
\begin{aligned}
&\frac {dU}{dt}(t,u) = \frac 12 \text{tr} \Big[(\bi uQ^{\frac 12})(\bi uQ^{\frac 12})^*D^2U(t,u)\Big]
+\<\bi \Delta u+\lambda \bi |u|^2u-\alpha u,DU(t,u)\>,  \\
&U(0,u) =\phi(u). \\
\end{aligned}
\right.
\end{equation*}

 In this section, we assume that $\phi \in C^3_b(\HH^1)\cap C^1_b(\HH)$, $\bs\ge2$, $Q^{\frac 12}\in\LL_2^\bs$, $u_0\in \mathbb H^\bs$,  $\alpha \in \HH^{2}$,  and $\sup\limits_{x\in\mathbb R}(\frac12 F_Q(x)- \alpha(x))\le -a$.
To remove the infinitesimal factor, we first eliminate the unbounded Laplacian operator and consider 
$V(t,v)=U(t,S(-t)v)$. Direct calculations show that $V$ satisfies 
\begin{equation}\label{Kol}
\left\{
\begin{aligned}
&\frac {dV}{dt}(t,v) = \frac 12 \text{tr} \Big[(S(t)(\bi (S(-t)v)Q^{\frac 12}))(S(t)(\bi (S(-t)v)Q^{\frac 12}))^*D^2V(t,v)\Big]\\
&\qquad\qquad\quad+\<\lambda \bi S(t)(|S(-t)v|^2(S(-t)v)),DV(t,v)\>-\<S(t)\alpha S(-t)v, DV(t,v)\>,  \\
&V(0,v) =\phi(v). \\
\end{aligned}
\right.
\end{equation}
Now, it can be shown that the functions $U$ and $V$ have the same regularity as the initial data $\phi$. 
Proposition \ref{exp-u} is the key to proving the following regularity result, which generalizes the case of 
Lipschitz drift operators in \cite{BD06}. 

\begin{lm}\label{UV-reg}
The functions $U$ and $V$ are continuous in time with values in \\$C^3(\HH^1)\cap C^1(\HH)$.
\end{lm}

\begin{proof}
	Differentiating $U$, we obtain for $h\in \HH$,
	\begin{align*}
	\<DU(t,u_0),h\>=\E\left[\<D\phi(u(t,u_0)),\eta^h(t)\>\right],
	\end{align*}
	where 
	\begin{equation*}
	\left\{
	\begin{aligned}
	&d\eta^h = \bi \Delta \eta^hdt+\bi \lambda\left(|u|^2\eta^h+2\Re(\bar u\eta^h)u \right)dt -\alpha \eta^hdt
	+\bi \eta^h dW(t)\\
	&\eta^h(0) =h. \\
	\end{aligned}
	\right.
	\end{equation*}
	The It\^o formula yields that 
	\begin{align*}
	\frac 12\|\eta^h(t)\|^2
	&=\frac 12\|h\|^2+\int_0^t \Big(\<\eta^h, \bi \Delta \eta^h\>+
	\<\eta^h, \bi \lambda\left(2|u|^2\eta^h+u^2\overline{\eta^h}\right)\>
	-\<\eta^h,\alpha \eta^h\>\Big)dr\\
	&\quad+\int_0^t \<\eta^h,	-\bi \eta^h dW(r) \>
	+\int_0^t\frac 12 \text{tr}[(-\bi \eta^h Q^{\frac 12})(-\bi \eta^h Q^{\frac 12})^*]dr\\
	&\le \frac 12\|h\|^2+\int_0^t
	\<\eta^h, \bi \lambda u^2\overline{\eta^h} \>dr-a\int_0^t
	\|\eta^h\|^2dr.
	\end{align*}
	By the Gronwall inequality, we obtain 
	\begin{align}\label{eta-h}
	\|\eta^h(t)\|^2\le \exp(-2at) \exp\left(\int_0^T2\|u\|^2_{L_{\infty}}dr\right)\|h\|^2. 
	\end{align}	
	Then taking expectation combined with  Proposition \ref{exp-u} yields that
	\begin{align*}
	\E\left[\sup_{t\in [0,T]}\|\eta^h(t)\|^2\right]
	\le C(u_0)\|h\|^2.
	\end{align*}
	Applying the  It\^o formula to $\|\eta^h\|^p,\; p\ge 2$, we get
	\begin{align*}
	\E\left[\sup_{t\in [0,T]}\|\eta^h(t)\|^p\right]
	\le C(u_0)\|h\|^p,
	\end{align*}	
	which implies that
	\begin{align*}
	\|DU(t,u_0)\|_{\LL(\HH, \R)}\le C(u_0)\|\phi\|_{C_b^1(\HH)}.
	\end{align*}
	Similarly, 
	\begin{align*}
	D^2U(t,u_0)\cdot(h,h)
	=\E\left[D^2\phi(u(t,u_0))\cdot(\eta^h(t),\eta^h(t))
	+D\phi(u(t,u_0))\cdot \xi^h(t) \right]
	\end{align*}			
	with 		
	\begin{equation*}
	\left\{
	\begin{aligned}
	&d\xi^h = \bi \Delta \xi^h dt+\bi \lambda\left(4\Re(\bar u \eta^h)\eta^h+2|\eta^h|^2 u\right)dt\\
	&\qquad\quad+\bi \lambda\left(|u|^2\xi^h+2\Re(\bar u\xi^h)u \right)dt -\alpha \xi^hdt
	+\bi \xi^h dW(t),\\
	&\xi^h(0) =0. \\
	\end{aligned}
	\right.
	\end{equation*}
	Again by the It\^o formula, we obtain 
	\begin{align*}
	&\frac 12\|\xi^h(t)\|^2\\
	&=\int_0^t\Bigg(\<\xi^h ,\bi \Delta \xi^h \>+\< \xi^h, \bi \lambda\left(4\Re(\bar u \eta^h)\eta^h+2|\eta^h|^2 u\right)\>\\
	&\qquad+<\xi^h, \bi \lambda\left(|u|^2\xi^h+2\Re(\bar u\xi^h)u \right)>-\<\xi^h,  \alpha\xi^h\>\Bigg)dr\\
	&\quad+\int_0^t \<\xi^h,-\bi \xi^h dW(r)\>
	+\frac 12\int_0^t \text{tr}[(-\bi \xi^h Q^{\frac 12})(-\bi \xi^h Q^{\frac 12})^*]dr	\\
	&\le\int_0^t\Bigg(\< \xi^h, \bi \lambda\left(4\Re(\bar u \eta^h)\eta^h+2|\eta^h|^2 u\right)\>+<\xi^h, \bi \lambda2\Re(\bar u\xi^h)u>\Bigg)dr
	-a\int_0^t\|\xi^h\|^2dr\\
	&\le \int_0^t-(a-\epsilon)\|\xi^h\|^2+2\|u\|^2_{L^{\infty}}\|\xi^h\|^2dr+\int_0^t C(\epsilon)\|u\|^2\|\eta^h\|^2\|\nabla \eta^h\|^2dr.
	\end{align*}
	Then the Gronwall inequality and the charge evolution of $u$ imply that 			
	\begin{align*}	
	\|\xi^h(t)\|^2\le C\exp\left(\int_0^T4\|u\|_{L^{\infty}}^2dr\right)
	\int_0^Te^{-2ar}\|u_0\|^2\|\eta^h\|^2\|\nabla \eta^h\|^2dr.
	\end{align*}
		After taking expectation, by Proposition \ref{exp-u}, we have 
		\begin{align*}
		\E\left[\sup_{t\in[0,T]}\|\xi^h(t)\|^2\right]
		&\le C(u_0) \sqrt[4]{\sup_{t\in[0,T]}\E\left[\|\eta^h(t)\|^8\right]}
		\sqrt[4]{\sup_{t\in[0,T]}\E\left[\|\nabla \eta^h(t)\|^8\right]}.	
		\end{align*}
		We need to show $\E[\|\nabla \eta^h\|^p]< \infty$. For simplicity, we give the proof for $p=2$.
		The proof of $p>2$ is similar to the previous arguments for $p>2$ in the a priori estimate of $u$ in the  $\HH^1$-norm.
	The It\^o formula, integration by parts, and the  Gagliardo--Nirenberg and Young inequalities show that 
	\begin{align*}
	&\frac 12\|\nabla \eta^h(t)\|^2\\
	&= \frac 12\|\nabla h\|^2+\int_0^t \<-\Delta \eta^h,\bi \Delta \eta^h+\bi \lambda (|u|^2\eta^h+2\Re(\bar u\eta^h)u)-\alpha \eta^h\>dr\\
	&\quad+\int_0^t\<-\Delta \eta^h, \bi \eta^hdW(r)\>
	+\frac 12\int_0^t\text{tr}[(-\bi \nabla(\eta^hQ^{\frac 12}) )(-\bi \nabla(\eta^hQ^{\frac 12}) )^*]dr\\
	&= \frac 12\|\nabla h\|^2+\int_0^t\lambda  \<\nabla  \eta^h,\bi 2\Re(\bar u \nabla u) \eta^h+\bi 2\Re(\bar u \nabla \eta^h)u
	+\bi 2\Re(\nabla \bar u \eta^h)u+\bi 2\Re( \bar u \eta^h)\nabla u\>dr\\
	&\quad-a \int_0^t\|\nabla \eta^h(t)\|^2+\int_0^t\<\nabla \eta^h,  \eta^h (\frac  12\nabla F_Q-\nabla \alpha) \>dr
	+\int_0^t\<\nabla \eta^h,\bi \eta^h d\nabla W(r))\>\\
	&\quad +\int_0^t\frac 12\sum_{k}\<\eta^h\nabla Q^{\frac 12}e_k,\eta^h\nabla Q^{\frac 12}e_k\>dr.
	\end{align*}
	The Gronwall inequality implies that for $s\in[0,t]$,
\begin{align*}
	\|\nabla \eta^h(t)\|^2
	&\le \exp\Big(-2(a-\epsilon)t+\int_0^T4\|u\|_{L^{\infty}}^2dr\Big)\Big( \|\nabla h\|^2+C(\epsilon, \alpha, Q)\int_{0}^T(\|\Delta u\|\|\nabla u\|^2\\
	&\qquad \|u\|+1)
	\|\eta^h\|^2dr
	+\sup_{s\in[0,t]}\Big|\int_0^s\<\nabla \eta^h,\bi \eta^h d\nabla W(r))\>\Big|\Big)
	\end{align*}
Then taking expectation, combined with Corollary \ref{cor-char}, Propositions \ref{hs} and  \ref{exp-u}, the estimation \eqref{eta-h}, and the Burkholder--Davis--Gundy, H\"older, and Young inequalities, leads that for $\frac 1p +\frac 1q=1$, $1<q<2$,
 	\begin{align*}
	&\E[\|\nabla \eta^h(t)\|^2]\\
	&\le e^{-2(a-\epsilon)t}\Big\|\exp\Big(\int_0^T4\|u\|_{L^{\infty}}^2dr\Big)\Big\|_{L^p(\Omega)}\Bigg(\Big\|\sup_{s\in[0,t]}\Big|\int_0^s\<\nabla \eta^h,\bi \eta^h d\nabla W(r)\>\Big|\Big\|_{L^q(\Omega)}\\
	&\qquad\qquad\qquad\quad\|\nabla h\|^2+C(\epsilon, \alpha, Q)\int_{0}^T\Big\|(\|\Delta u\|\|\nabla u\|^2\|u\|+1)
	\|\eta^h\|^2\Big\|_{L^q(\Omega)}dr
	\Bigg)\\
	&\quad
	\le e^{-2(a-\epsilon)t}C(p,\alpha, Q,u_0)\Big(
	\sqrt[q]{\E\Big[\Big(\int_0^t\|\nabla \eta^h\|^2dr\Big)^{\frac q2}\sup_{s\in[0,T]}\| \eta^h(s)\|^q\Big]}+\|\nabla h\|^2+\|h\|^2\Big)\\
	&\quad
	\le e^{-2(a-\epsilon)t}C(p, \alpha, Q,u_0)\Big(
	\int_0^t\epsilon\E\Big[\|\nabla \eta^h\|^2\Big]dr
	+C(\epsilon)(\|\nabla h\|^2+\|h\|^2)\Big)
	\end{align*}
Applying again the Gronwall inequality, we get
	\begin{align*}
	&\E\left[\|\nabla \eta^h(t)\|^2\right]\le C(\epsilon, \alpha, Q, u_0)(\|h\|^2+\|\nabla h\|^2).
	\end{align*}	
	Similar arguments yield that for any $p\ge 2$,
	\begin{align*}
	\left\|\nabla \eta^h(t)\right\|_{L^p(\Omega; \HH)}\le C(u_0)(\|h\|+\|\nabla h\|).
	\end{align*}
	Then we conclude that 
	\begin{align*}
	\E\left[\sup_{t\in[0,T]}\|\xi^h(t)\|^p\right]
	&\le C(u_0)\|h\|^p\|\nabla h\|^p,
	\end{align*}
	which implies that 
	\begin{align*}
	\|D^2U(t,u_0)\|_{\LL(\HH^1\times\HH^1;\R)}\le C(u_0)\max(\|\phi\|_{C^2_b(\HH^1)},\|\phi\|_{C^1_b(\HH)} ).
	\end{align*}
	For the function $V(t,v)=U(t,S(-t)v)$,	
	we have 
	\begin{align*}
	\<DV(t,u_0), h\>_\HH
	=\E\left[\<D\phi(u(t,S(-t)u_0)),\eta^h\>\right],
	\end{align*}
	and 
	\begin{align*}
	DV(t,u_0)\cdot(h,h)
	=\E\left[D^2\phi(u(t,S(-t)u_0))\cdot(\eta^h(t),\eta^h(t))
	+D\phi(u(t,S(t)u_0))\cdot \xi^h(t) \right].
	\end{align*}
	The  unitarity of  $S(t)$, i.e., $\|S(t)u_0\|_{\HH^\bs}=\|u_0\|_{\HH^\bs}$, $\bs\in N$, combined with previous arguments finishes the proof. Similar arguments yield that $U$ and $V$ belong to $C^3(\HH^1)$.
\end{proof}

\begin{rk}
	The above procedures imply the global existence of variational solutions of stochastic NLS equations, which in turn gives the theoretical support to why the phase flow, in any finite time,  preserves the symplectic structure when $\alpha=\frac 12 F_Q$  and   the conformal symplectic structure when $\alpha=a+\frac 12F_Q$ (see, e.g., \cite{CH16, HWZ17}).
\end{rk}

Based on the estimations in Lemma \ref{UV-reg} and the corresponding Kolmogorov equation, we have the following weak convergence
result.

\begin{tm}\label{weak}
	Assume that $\alpha \in \HH^{4}$, $\|Q^{\frac 12}\|_{\LL_2^4}< \infty$, and   $u_0\in \HH^4$. For any $\phi\in C_b^3(\HH^1)\cap C_b^1(\HH)$, there exists a positive 
	constant $C=C(\alpha, Q,u_0,\phi)$ such that 
	\begin{align}\label{wk}
	\left|\E\left[\phi(u(T))\right]- \E\left[\phi(u_M)\right]\right|\le C \tau.
	\end{align}
\end{tm}

	We aim to give the representation formula of the  weak error and  split 
	 $\E\left[\phi(u(T))\right]-\E\left[\phi(u_M)\right]$
	as  follows:
	\begin{align*}
	\E\left[\phi(u(T))\right]-\E\left[\phi(u_M)\right]
	=\E\left[\phi(u(T))\right]-\E\left[\phi(u_\tau(T))\right]
	+\E	\left[\phi(u_\tau(T))\right]-\E\left[\phi(u_M)\right].
	\end{align*}	
The following lemmas show that the estimate \eqref{wk} holds.

	\begin{lm}\label{wk-ut}
		Assume that  $\alpha \in \HH^{2}$, $\|Q^{\frac 12}\|_{\LL_2^2}< \infty $ and $u_0\in \HH^2$.  For any $\phi\in C_b^3(\HH^1)\cap C_b^1(\HH)$, there exists a positive 
		constant $C=C(\alpha, Q,u_0,\phi)$ such that 
		\begin{align*}
		\left|\E\left[\phi(u(T))\right]- \E\left[\phi(u_\tau(T))\right]\right|\le C \tau.
		\end{align*}
		\end{lm}
		
	\begin{proof}	
	First  we split the error by the local arguments as
	follows: 
	\begin{align*}
	\E\left[\phi(u_\tau(T))\right]-\E\left[\phi(u(T))\right]
	= \sum_{k=0}^{M-1} \Big(\E [V(T-t_{k+1},v_\tau(t_{k+1}))]-\E [V(T-t_{k},v_\tau(t_{k}))]\Big),
	\end{align*}
	where $v_\tau(t)=S(T-t)u_\tau(t)$.
	The definition of $u_\tau $ yields that 
	\begin{align*}
	S(T-t_{k+1})u_\tau(t_{k+1})&=S(T-t_{k})u_\tau(t_{k})+\int_{t_k}^{t_{k+1}}S(T-t)\bi |u_{\tau}^D(t)|^2u_{\tau}^D(t)dt\\
	&\quad +\int_{t_k}^{t_{k+1}}S(T-t)\bi u_\tau^S(t)dW(t)-\int_{t_k}^{t_{k+1}}S(T-t)\alpha u_\tau^S(t)dt.
	\end{align*}
	With the help of the  Kolmogorov equation \eqref{Kol},
	the mean value theorem, and the It\^o formula, we get 
	
	\begin{align*}
	&V(T-t_{k+1},v_\tau(t_{k+1}))-V(T-t_{k},v_\tau(t_{k}))\\
	&=
	-\int_{t_k}^{t_{k+1}} \frac {dV}{dt}dt+\int_{t_k}^{t_{k+1}} \frac 12 
	 \text{tr} \Big[(S(T-t)
	 (\bi (S(-T+t)v_\tau) Q^{\frac 12}))
	 (S(T-t)(\bi (S(-T+t)\\
	 &\qquad v_\tau) Q^{\frac 12}))^*D^2V(T-t,v_\tau)\Big]dt
	 -\int_{t_k}^{t_{k+1}} \Big\<S(T-t)\alpha  S(-T+t)v_\tau, DV(T-t,v_\tau)\Big\>dt\\
	 &\quad +\int_{t_k}^{t_{k+1}} \Big\<\lambda \bi S(T-t)
	 (|u_\tau^D|^2 u_\tau^D), \int_{0}^1 DV(T-t_k, v_\tau(t_k)
	 +\theta \int_{t_k}^{t_{k+1}}S(T-t)\bi |u_\tau^D(s)|^2u_\tau^D(s)ds )d\theta\Big\> dt\\
	  &\quad +\int_{t_k}^{t_{k+1}}\Big\<S(T-t)\bi u_\tau^S(t)dW(t), DV(T-t,v_\tau) \Big\>\\
	 &=\int_{t_k}^{t_{k+1}} \Big\<\lambda \bi S(T-t)
	 (|u_\tau^D|^2 u_\tau^D), \int_{0}^1 DV(T-t_k, v_\tau(t_k)
	 +\theta \int_{t_k}^{t_{k+1}}S(T-t)\bi |u_\tau^D(s)|^2u_\tau^D(s)ds )d\theta\Big\> dt\\
	 &\quad -\int_{t_k}^{t_{k+1}} \Big\<\lambda \bi S(T-t)
	 (|u_\tau|^2 u_\tau),DV(T-t,v_\tau)\Big\>dt
	  +\int_{t_k}^{t_{k+1}}\Big\<S(T-t)\bi u_\tau^S(t)dW(t), DV(T-t,v_\tau) \Big\>.
	\end{align*}
	
	Then taking expectation shows  that
	\begin{align*}
	&\E \Big[V(T-t_{k+1},v_\tau(t_{k+1}))-V(T-t_{k},v_\tau(t_{k}))\Big]\\
	&= \E \Bigg[\int_{t_k}^{t_{k+1}}\Big\<S(T-t)\bi |u_\tau^D|^2u_\tau^D,\int_{0}^1DV(T-t_k, v_\tau(t_k)\\
	&\quad +\theta \int_{t_k}^{t_{k+1}}S(T-t)\bi |u_\tau^D(s)|^2u_\tau^D(s)ds) d\theta -DV(T-t, v_\tau(t_k))\Big\>dt\\
	&\quad -\int_{t_k}^{t_{k+1}} \Big\<\lambda \bi S(T-t)
	(|u_\tau|^2 u_\tau)- \bi S(T-t)
	(|u_\tau^D|^2 u_\tau^D) ,DV(T-t,v_\tau(t_k))\Big\>dt\\
	&\quad -\int_{t_k}^{t_{k+1}} \Big\<\lambda \bi S(T-t)
	(|u_\tau|^2 u_\tau) ,DV(T-t,v_\tau(t))- DV(T-t,v_\tau(t_k))\Big\>dt\Bigg]\\
	& :=\E [\mathcal W_1+\mathcal W_2+\mathcal W_3].
	\end{align*}
Since $u_\tau $ and $u_\tau^D$  are both predictable, combining them with the continuity of $D^2V$ and $D^3V$ and the expansion of $DV$, we get  for some $b_1>0$,
\begin{align*}
\E [\mathcal W_1]
&\le C(u_0)\tau^2\sup_{t\in [0,T]}\sqrt{\E [\|u_\tau^D(t)\|^{12}_{\HH^1}]}\\
&+ \E \Bigg[\int_{t_k}^{t_{k+1}}\Big\<S(T-t)\bi |u_\tau^D|^2u_\tau^D,DV(T-t_k, v_\tau(t_k))-DV(T-t, v_\tau(t_k))\Big\>dt\\
&\le  C(u_0)\tau^2\sup_{t\in [0,T]}\sqrt{\Big(\E [\|u_\tau^D(t)\|^{12}_{\HH^1}] +\E [\|u_\tau^D(t)\|^2_{\HH^1}] \Big)} \le Ce^{-b_1t_k}\tau^2.
\end{align*}
The cubic difference formula yields that
\begin{align*}
\mathcal W_2
&=-\int_{t_k}^{t_{k+1}} \Big\<\lambda \bi S(T-t)\Big(
(|u_\tau|^2 +|u_\tau^D|^2) (u_\tau-
 u_\tau^D)  \Big) ,DV(T-t,v_\tau(t_k))\Big\>dt\\
 &\quad -\int_{t_k}^{t_{k+1}} \Big\<\lambda \bi S(T-t)\Big(
  u_\tau u_\tau^D(\overline{u_\tau} -\overline{u_\tau^D} )
    \Big) ,DV(T-t,v_\tau(t_k))\Big\>dt:=\mathcal W_{21}+\mathcal W_{22}.
\end{align*}
The estimations of $\mathcal W_{21}$ and $\mathcal W_{22}$ are similar, we only give the estimate of first term.
The expressions of $u_\tau^D$ and $u_\tau$ yield that 
\begin{align*}
\mathcal W_{21}
&=\int_{t_k}^{t_{k+1}} \Big\<\lambda \bi S(T-t)\Big(
(|u_\tau|^2 +|u_\tau^D|^2) (u_\tau^D(t)-u_\tau^D(t_k)
) \Big) ,DV(T-t,v_\tau(t_k))\Big\>dt\\
&\quad+
\int_{t_k}^{t_{k+1}} \Big\<\lambda \bi S(T-t)\Big(
(|u_\tau|^2 +|u_\tau^D|^2) (u_\tau^D(t_k)-u_\tau(t)
) \Big) ,DV(T-t,v_\tau(t_k))\Big\>dt\\
&=
\int_{t_k}^{t_{k+1}} \Big\<\lambda \bi S(T-t)\Bigg(
(|u_\tau|^2 +|u_\tau^D|^2) \Big( (S(t-t_k)-I)u_\tau(t_k)
 +\int_{t_k}^t S(t-t_k) \\ 
 &|u_\tau^D(s)|^2u_\tau^D(s)ds
 \Big) \Bigg) ,DV(T-t,v_\tau(t_k))\Big\>dt
-\int_{t_k}^{t_{k+1}} \Big\<\lambda \bi S(T-t)\Bigg(
(|u_\tau|^2 +|u_\tau^D|^2)\\
&\qquad\quad \Big((S(t_{k+1}-t_k)-I)u_\tau(t_k)
+\int_{t_k}^{t_{k+1}}
S(t_{k+1}-s)
|u_\tau^D(s)|^2u_\tau^D(s)ds\\
&\qquad+\int_{t_k}^{t}
\bi u_\tau^SdW(s) -\int_{t_k}^{t} \alpha u_\tau^S(s)ds
\Big) \Bigg) ,DV(T-t,v_\tau(t_k))\Big\>dt.
\end{align*}	
Then the independence of the Wiener process and the property of the  stochastic integral, together with the property of $S(t)$, 
the boundedness of $DV$, and the Gagliardo--Nirenberg inequality, imply that
\begin{align*}
&\E [\mathcal W_{21}]\\
&\le C(u_0)\tau^2\sup_{t\in[t_k,t_{k+1}]}\sqrt{\E \Big[(\|u_\tau\|^2+\|u_\tau^D\|^2)\Big(1+\|u_\tau\|_{\HH^2}^4
+\|u_\tau^D\|_{\HH^2}^4+\|u_\tau\|_{\HH^1}^{12}+\|u_\tau^D\|_{\HH^1}^{12}\Big)\Big]}\\
&\le Ce^{-at_k}\tau^2. 
\end{align*}
Thus we can obtain $
\E [\mathcal W_{2}]\le Ce^{-at_k}\tau^2$ .

 The boundedness of $D^2V$, the continuity of $u_\tau$ and $v_\tau$ in the local interval, the property of the stochastic integral,  and Lemma \ref{dis-char} and \ref{dis-char-cn} imply for $b_2>0$ that 
\begin{align*}
\E[\mathcal W_3]
&\le C\sqrt{\E\Big[\Big(\int_{t_k}^{t_{k+1}}(\|u_\tau(t)\|_{\HH^1}^2+\|u_\tau(t_k)\|_{\HH^1}^2)\|u_\tau(t_k)-u_\tau(t)\|_{\HH^1}
\|v_\tau(t)-v_\tau(t_k)\|_{\HH^1}dt\Big)^2\Big]}\\
&\quad-\E\Bigg[\Big\<\int_{t_k}^{t_{k+1}}\int_{t_k}^t D^2V(T-s,v_\tau(s)) dv_\tau(s)ds,\bi \lambda S(T-t)(|u_\tau(t_k)|^2u_\tau(t_k))\Big\> dt \Bigg]\\
&\quad+C \E\Bigg[\int_{t_k}^{t_{k+1}} \Big\|\int_{0}^1 D^2V(T-t, v_\tau(t_k)
+\theta \int_{t_k}^{t_{k+1}}S(T-t)\bi |u_\tau^D(s)|^2u_\tau^D(s)ds )d\theta\Big\|
\\
&\quad \times \|u_\tau(t_k)\|_{\HH^1}^3  \|\int_{t_k}^{t_{k+1}}S(T-t)\bi |u_\tau^D(s)|^2u_\tau^D(s)ds\|_{\HH^1}  dt\Bigg]\\
&\le Ce^{-b_2t_k}\tau^2.
\end{align*}
The estimations of $\mathcal W_i$, $i=1,2,3,$ yield that 
\begin{align*}
\E\left[\phi(u(T))\right]-\E\left[\phi(u_\tau(T))\right]\le
C\sum_{k=0}^{M-1}e^{-\min(a,b_1,b_2)t_k}\tau^2\le C\tau.
\end{align*}

\end{proof}

Next, we deal with the term $\E	\left[\phi(u_\tau(T))\right]-\E\left[\phi(u_M)\right]$.

\begin{lm}\label{wk-um}
	Assume that $\alpha \in \HH^{4}$, $\|Q^{\frac 12}\|_{\LL_2^4}< \infty$, and   $u_0\in \HH^4$. For any $\phi\in C_b^3(\HH^1)\cap C_b^1(\HH)$, there exists a positive 
	constant $C=C(\alpha, Q,u_0,\phi)$ such that 
	\begin{align*}
	\left|\E\left[\phi(u_\tau(T))\right]- \E\left[\phi(u_M)\right]\right|\le C \tau.
	\end{align*}
\end{lm}

\begin{proof}
By the damped effect, we  obtain for any $v, w \in \HH $,
\begin{align*}
\|\Phi_{k}^S(v-w)\|\le Ce^{-a\tau}\|v-w\|.
\end{align*}	
Then the total error are divided as follows:
\begin{align*}
&|\phi(u_\tau(T))- \phi(u_M)| \le C\|u_\tau(T)-u_M\|\\
&\le C\sum_{k=0}^{M-1} \left\|\prod_{j=1}^{M-k-1}\big(\Phi_{M-j}^S\widehat \Phi_{M-j}^D\big)\Phi_{k}^S(\widehat \Phi_{k}^D-\Phi_{k}^D) \prod_{l=0}^{k-1} \big(\Phi_{k-1-l}^S \Phi_{k-1-l}^D\big) u_0\right\|\\
&\le C\sum_{k=0}^{M-1}\left\|\prod_{j=1}^{M-k-1}\big(\Phi_{M-j}^S\widehat\Phi_{M-j}^D\big)\Phi_{k}^S(\widehat \Phi_{k}^D-\Phi_{k}^D)u_\tau^D(t_k)\right\|.
\end{align*}
By the stability of $u_\tau^D$ in $\HH^4$, we have 
\begin{align*}
&\Big\|(\widehat \Phi_{k}^D-\Phi_{k}^D)u_\tau^D(t_k)\Big\|
\le \Big\|\Big(S(t_{k+1}-t_k)-S_{\tau}\Big)u_\tau^D(t_k)\Big\|\\
&\quad +\int_{t_k}^{t_{k+1}}\Big\| S(t_{k+1}-s)|u_\tau^D(s)|^2u_\tau^D(s)-\tau T_{\tau} \frac {|\widehat u_\tau^D(t_k)|^2+|\widehat u^D_{k+1}|^2}{2}\widehat u^D_{k+\frac 12}\Big\|dt\\
&\le C\tau^2\|u_\tau^D(t_k)\|_{\HH^4}+C\tau\Big(\int_{t_k}^{t_{k+1}}\|u_\tau^D(s)\|^5_{\HH^1}ds+\tau\|u_\tau^D(t_k)\|^5_{\HH^1}+\tau\|\widehat u_{k+1}^D\|_{\HH^1}^5\Big),
\end{align*}
where $S_{\tau t}= \frac {1+\frac \bi 2 \Delta \tau}{1-\frac \bi 2 \Delta \tau}$, $T_{\tau t}= \frac 1 {1-\frac \bi 2 \Delta \tau}$.
After taking expectation,  we see that charge evolutions and the continuous dependence on initial data  of $\widehat \Phi_{k}^D,  \Phi_{k}^D, \Phi_{k}^S$, $k\in Z_M$, 
together with the uniform boundedness of $\widehat u_{k}$ and $u^D_{t_k}$ in Proposition \ref{dis-hs},  and Lemma \ref{dis-char-cn}, imply that
\begin{align*}
\E\Big[|\phi(u_\tau(T))- \phi(u_M)| \Big]
&\le 
C\sum_{k=0}^{M-1}e^{-a(T-t_{k+1})}\tau^2\le C\tau.
\end{align*}
   
\begin{rk}
Since we discretize the semigroup $S(\tau)$ by $S_{\tau}$, we need the same high regularity requirement  on $u_0$ as in \cite{BD06} to get a weak order result. 
This approach to analyze weak order of numerical scheme is also available for the conservative stochastic NLS equation ($\alpha = \frac 12F_Q $) and other cases, such as $\|\alpha\|_{\HH^4}< \infty$ and for more general test functions with polynomial growths, i.e., $\phi \in C_p^3(\HH^1)\cap C_p^1(\HH)$.
\end{rk}

\end{proof}

\appendix

\bibliographystyle{amsplain}
\bibliography{bib}

\providecommand{\bysame}{\leavevmode\hbox to3em{\hrulefill}\thinspace}
\providecommand{\MR}{\relax\ifhmode\unskip\space\fi MR }
\providecommand{\MRhref}[2]{%
  \href{http://www.ams.org/mathscinet-getitem?mr=#1}{#2}
}
\providecommand{\href}[2]{#2}
\begin{thebibliography}{10}

\bibitem{AC16}
R.~Anton and D.~Cohen, \emph{Exponential integrators for stochastic
  {S}chr\"odinger equations driven by ito noise}, J. Comput. Math., to appear.

\bibitem{ACLW16}
R.~Anton, D.~Cohen, S.~Larsson, and X.~Wang, \emph{Full discretization of
  semilinear stochastic wave equations driven by multiplicative noise}, SIAM J.
  Numer. Anal. \textbf{54} (2016), no.~2, 1093--1119. \MR{3484400}

\bibitem{BCI95}
O.~Bang, P.~L. Christiansen, F.~If, K.~{O}~. Rasmussen, and Y.~B. Gaididei,
  \emph{White noise in the two-dimensional nonlinear {S}chr\"odinger equation},
  Appl. Anal. \textbf{57} (1995), no.~1-2, 3--15. \MR{1382938}

\bibitem{BRZ16}
V.~Barbu, M.~R\"ockner, and D.~Zhang, \emph{Stochastic nonlinear
  {S}chr\"odinger equations}, Nonlinear Anal. \textbf{136} (2016), 168--194.
  \MR{3474409}

\bibitem{BAK16}
S.~Becker, A.~Jentzen, and P.~E. Kloeden, \emph{An exponential
  {W}agner-{P}laten type scheme for {SPDE}s}, SIAM J. Numer. Anal. \textbf{54}
  (2016), no.~4, 2389--2426. \MR{3534472}

\bibitem{CHL16a}
Y.~Cao, J.~Hong, and Z.~Liu, \emph{Approximating stochastic evolution equations
  with additive white and rough noises}, SIAM J. Numer. Anal. \textbf{55}
  (2017), no.~4, 1958--1981. \MR{3686802}

\bibitem{CH16}
C.~Chen and J.~Hong, \emph{Symplectic {R}unge--{K}utta {S}emidiscretization for
  {S}tochastic {S}chr\"odinger {E}quation}, SIAM J. Numer. Anal. \textbf{54}
  (2016), no.~4, 2569--2593. \MR{3542010}

\bibitem{CHP16}
C.~Chen, J.~Hong, and A.~Prohl, \emph{Convergence of a {$\theta$}-scheme to
  solve the stochastic nonlinear {S}chr\"odinger equation with {S}tratonovich
  noise}, Stoch. Partial Differ. Equ. Anal. Comput. \textbf{4} (2016), no.~2,
  274--318. \MR{3498984}

\bibitem{CHJ13}
S.~Cox, M.~Hutzenthaler, and A.~Jentzen, \emph{Local lipschitz continuity in
  the initial value and strong completeness for nonlinear stochastic
  differential equations}, arXiv:1309.5595.

\bibitem{CV10}
S.~Cox and J.~van Neerven, \emph{Convergence rates of the splitting scheme for
  parabolic linear stochastic {C}auchy problems}, SIAM J. Numer. Anal.
  \textbf{48} (2010), no.~2, 428--451. \MR{2646103}

\bibitem{CHL16b}
J.~Cui, J.~Hong, and Z.~Liu, \emph{Strong convergence rate of finite difference
  approximations for stochastic cubic {S}chr\"odinger equations}, J.
  Differential Equations \textbf{263} (2017), no.~7, 3687--3713. \MR{3670034}

\bibitem{CHLZ17}
J.~Cui, J.~Hong, Z.~Liu, and W.~Zhou, \emph{Strong convergence rate of
  splitting schemes for stochastic nonlinear {S}chr\"odinger equations},
  arXiv:1701.05680.

\bibitem{CHL17c}
\bysame, \emph{Stochastic symplectic and multi-symplectic methods for nonlinear
  {S}chr\"odinger equation with white noise dispersion}, J. Comput. Phys.
  \textbf{342} (2017), 267--285. \MR{3649275}

\bibitem{Dap06}
G.~Da~Prato, \emph{An introduction to infinite-dimensional analysis},
  Universitext, Springer-Verlag, Berlin, 2006, Revised and extended from the
  2001 original by Da Prato. \MR{2244975}

\bibitem{BD99}
A.~de~Bouard and A.~Debussche, \emph{A stochastic nonlinear {S}chr\"odinger
  equation with multiplicative noise}, Comm. Math. Phys. \textbf{205} (1999),
  no.~1, 161--181. \MR{1706888}

\bibitem{BD03}
\bysame, \emph{The stochastic nonlinear {S}chr\"odinger equation in {$H^1$}},
  Stochastic Anal. Appl. \textbf{21} (2003), no.~1, 97--126. \MR{1954077}

\bibitem{BD06}
\bysame, \emph{Weak and strong order of convergence of a semidiscrete scheme
  for the stochastic nonlinear {S}chr\"odinger equation}, Appl. Math. Optim.
  \textbf{54} (2006), no.~3, 369--399. \MR{2268663}

\bibitem{DO05}
A.~Debussche and C.~Odasso, \emph{Ergodicity for a weakly damped stochastic
  non-linear {S}chr\"odinger equation}, J. Evol. Equ. \textbf{5} (2005), no.~3,
  317--356. \MR{2174876}

\bibitem{Dor12}
P.~D{\"o}rsek, \emph{Semigroup splitting and cubature approximations for the
  stochastic {N}avier-{S}tokes equations}, SIAM J. Numer. Anal. \textbf{50}
  (2012), no.~2, 729--746. \MR{2914284}

\bibitem{Gyon99}
I.~Gy\"ongy, \emph{Lattice approximations for stochastic quasi-linear parabolic
  partial differential equations driven by space-time white noise. {I}},
  Potential Anal. \textbf{9} (1998), no.~1, 1--25. \MR{1644183}

\bibitem{GK03}
I.~Gy\"ongy and N.~Krylov, \emph{On the splitting-up method and stochastic
  partial differential equations}, Ann. Probab. \textbf{31} (2003), no.~2,
  564--591. \MR{1964941}

\bibitem{HWZ17}
J.~Hong, X.~Wang, and L.~Zhang, \emph{Numerical analysis on ergodic limit of
  approximations for stochastic {NLS} equation via multi-symplectic scheme},
  SIAM J. Numer. Anal. \textbf{55} (2017), no.~1, 305--327. \MR{3608750}

\bibitem{HJ14}
M.~Hutzenthaler and A.~Jentzen, \emph{On a perturbation theory and on strong
  convergence rates for stochastic ordinary and partial differential equations
  with non-globally monotone coefficients}, arXiv:1401.0295.

\bibitem{HJW13}
M.~Hutzenthaler, A.~Jentzen, and X.~Wang, \emph{Exponential integrability
  properties of numerical approximation processes for nonlinear stochastic
  differential equations}, Math. Comp \textbf{87} (2018), no.~311, 1353--1413.

\bibitem{JP16}
A.~Jentzen and P.~Pusnik, \emph{Exponential moments for numerical
  approximations of stochastic partial differential equations},
  arXiv:1609.07031.

\bibitem{Liu13a}
J.~Liu, \emph{Order of convergence of splitting schemes for both deterministic
  and stochastic nonlinear {S}chr\"odinger equations}, SIAM J. Numer. Anal.
  \textbf{51} (2013), no.~4, 1911--1932. \MR{3072234}

\end{thebibliography}
\end{document}